\documentclass[10pt]{amsart}
\usepackage{amsmath,amssymb,latexsym,esint,cite,mathrsfs}
\usepackage{verbatim,wasysym}
\usepackage[left=2.5cm,right=2.5cm,top=2.5cm,bottom=2.5cm]{geometry}
\usepackage{tikz,enumitem,graphicx, subfig, microtype, color}
\usepackage{epic,eepic}

\usepackage[colorlinks=true,urlcolor=blue, citecolor=red,linkcolor=blue,
linktocpage,pdfpagelabels, bookmarksnumbered,bookmarksopen]{hyperref}
\usepackage[hyperpageref]{backref}
\usepackage[english]{babel}

\numberwithin{equation}{section}

\newtheorem{thm}{Theorem}[section]
\newtheorem{lem}[thm]{Lemma}
\newtheorem{cor}[thm]{Corollary}
\newtheorem{Prop}[thm]{Proposition}

\newtheorem{Rem}[thm]{Remark}

\newcommand{\F}{\mathbb{F}}

\newcommand{\N}{\mathbb{N}}

\newcommand{\R}{\mathbb{R}}

\newcommand{\cal}{\mathcal}

\begin{document}
\baselineskip=14pt

\title[Critical Choquard equation with potential well]{On the critical Choquard equation with potential well}

\author[F. Gao]{Fashun Gao}
\author[Z. Shen]{Zifei Shen}
\author[M. Yang]{Minbo Yang$^*$}

\address{Fashun Gao, \newline\indent Department of Mathematics, Zhejiang Normal University, \newline\indent
	Jinhua 321004, People's Republic of China}
\email{fsgao@zjnu.edu.cn}

\address{Zifei Shen,\newline\indent Department of Mathematics, Zhejiang Normal University, \newline\indent
	Jinhua 321004, People's Republic of China}
\email{szf@zjnu.edu.cn}

\address{Minbo Yang, \newline\indent Department of Mathematics, Zhejiang Normal University, \newline\indent
	Jinhua 321004, People's Republic of China}
\email{mbyang@zjnu.edu.cn}


\thanks{Zifei Shen and Fashun Gao were partially supported by NSFC ($11671364$);\\
$^*$ Minbo Yang is the corresponding author, he was partially supported by NSFC ($11571317$) and ZJNSF(LY$15A010010$).}

\begin{abstract}
	In this paper we are interested in the following nonlinear Choquard equation
$$
-\Delta u+(\lambda V(x)-\beta)u
=\big(|x|^{-\mu}\ast |u|^{2_{\mu}^{\ast}}\big)|u|^{2_{\mu}^{\ast}-2}u\hspace{4.14mm}\mbox{in}\hspace{1.14mm} \mathbb{R}^N,
$$
where $\lambda,\beta\in\mathbb{R}^+$, $0<\mu<N$, $N\geq4$, $2_{\mu}^{\ast}=(2N-\mu)/(N-2)$ is the upper critical exponent due to the Hardy-Littlewood-Sobolev inequality and the nonnegative potential function $V\in \mathcal{C}(\mathbb{R}^N,\mathbb{R})$ such that $\Omega :=\mbox{int} V^{-1}(0)$ is a nonempty bounded set with smooth boundary. If $\beta>0$ is a constant such that the operator $-\Delta +\lambda V(x)-\beta$ is non-degenerate, we prove the existence of ground state solutions which localize near the potential well int $V^{-1}(0)$ for $\lambda$ large enough and also characterize the asymptotic behavior of the solutions as the parameter $\lambda$ goes to infinity. Furthermore, for any $0<\beta<\beta_{1}$, we are able to find the existence of multiple solutions by the Lusternik-Schnirelmann category theory, where $\beta_{1}$ is the first eigenvalue of $-\Delta$ on $\Omega$ with Dirichlet boundary condition.
\end{abstract}

\maketitle

\begin{center}
	\begin{minipage}{8.5cm}
		\small
		\tableofcontents
	\end{minipage}
\end{center}
%

\section{Introduction and main results}
In this paper we are concerned with the existence of solutions of the Choquard type equation
\begin{equation}\label{Nonlocal.S1}
-\Delta u+ V(x)u=\big(|x|^{-\mu}\ast |u|^{q}\big)|u|^{q-2}u,\hspace{4.14mm} \mbox{in}\ \mathbb{R}^N,
\end{equation}
where $N\geq4$, $0<\mu<N$ and $V(x)$ is the external potential. This type of equation goes
back to the description of the quantum theory of a polaron at rest by S. Pekar in 1954 \cite{Ps}
and the modeling of an electron trapped
in its own hole in 1976 in the work of P. Choquard, as a certain approximation to Hartree-Fock theory of one-component
plasma \cite{L1}. In some particular cases, this equation is also known as the Schr\"{o}dinger-Newton equation, which was introduced by Penrose in his discussion on the selfgravitational collapse of a quantum mechanical wave function \cite{Pe}.

In last decades, a great deal of mathematical efforts have been devoted to the study of existence, multiplicity and properties of solutions of the nonlinear Choquard equation \eqref{Nonlocal.S1}. For constant potentials, if $N=3$, $q=2$ and $\mu=1$, the existence of ground states of equation \eqref{Nonlocal.S1} was obtained in \cite{L1, Ls} by variational methods. Involving the qualitative properties of the ground stats, the uniqueness was proved in \cite{L1} and the nondegeneracy was established in \cite{Len, WW}. For equation \eqref{Nonlocal.S1} with general $q$ and $\mu$, the regularity, positivity, radial symmetry and decay property of the ground states were proved in \cite{CCS1, MZ, MS1}. Moreover, the existence of positive ground states under the assumptions of Berestycki-Lions type in \cite{MS2}. For the existence of sign-changing solutions of the nonlinear Choquard equation, we refer the readers to the references \cite{CS, GMS, GS}.
 For nonconstant potentials, if $V$ is a continuous periodic function with $\inf_{\R^3} V(x)> 0$, noticing that the nonlocal term is invariant under translation, we can obtain easily the existence result by applying the Mountain Pass Theorem.
 If $V$ changes sign and $0$ lies in the gap of the spectrum of the Schr\"{o}dinger operator $-\Delta +V$, the problem is strongly indefinite, and the existence of solution for $q=2$ was considered in \cite{BJS} and  the existence of infinitely many geometrically distinct weak solutions in\cite{AC}. 
 
  If the nonlinear Choquard equation is equipped with deepening potential well of the form $\lambda a(x)+1$ where $a(x)$ is a  nonnegative continuous function such that $\Omega =$ int $(a^{-1}(0))$ is a non-empty bounded open set with smooth boundary. Moreover, suppose that $\Omega$ has $k$ connected components, more precisely,
\begin{equation} \label{a0}
\Omega =\bigcup_{j=1}^{k}\Omega_j
\end{equation}
with
\begin{equation} \label{a00}
\mbox{dist}( \Omega_i, \Omega_j)>0  \quad \mbox{for} \quad i \neq j,
\end{equation}
the existence and multiplicity of multi-bump shaped solution in \cite{ANY}.

We need to point out that all the existing results for the nonlinear Choquard equation \eqref{Nonlocal.S1} require that the exponent $q$ satisfies
$$
\frac{2N-\mu}{N}<q<\frac{2N-\mu}{N-2}.
$$
To understand why the range of $q$ make sense, it is necessary to recall the well-known Hardy-Littlewood-Sobolev inequality.
\begin{Prop}\label{HLS}
 (Hardy-Littlewood-Sobolev inequality). (See \cite{LL}.) Let $t,r>1$ and $0<\mu<N$ with $1/t+\mu/N+1/r=2$, $f\in L^{t}(\mathbb{R}^N)$ and $h\in L^{r}(\mathbb{R}^N)$. There exists a sharp constant $C(t,N,\mu,r)$, independent of $f,h$, such that
\begin{equation}\label{HLS1}
\int_{\mathbb{R}^{N}}\int_{\mathbb{R}^{N}}\frac{f(x)h(y)}{|x-y|^{\mu}}dxdy\leq C(t,N,\mu,r) |f|_{t}|h|_{r},
\end{equation}
where $|\cdot|_{q}$ for the $L^{q}(\mathbb{R}^{N})$-norm for $q\in[1,\infty]$. If $t=r=2N/(2N-\mu)$, then
$$
 C(t,N,\mu,r)=C(N,\mu)=\pi^{\frac{\mu}{2}}\frac{\Gamma(\frac{N}{2}-\frac{\mu}{2})}{\Gamma(N-\frac{\mu}{2})}\left\{\frac{\Gamma(\frac{N}{2})}{\Gamma(N)}\right\}^{-1+\frac{\mu}{N}}.
$$
In this case there is equality in \eqref{HLS1} if and only if $f\equiv(const.)h$ and
$$
h(x)=A(\gamma^{2}+|x-a|^{2})^{-(2N-\mu)/2}
$$
for some $A\in \mathbb{C}$, $0\neq\gamma\in\mathbb{R}$ and $a\in \mathbb{R}^{N}$.
\end{Prop}

By the Hardy-Littlewood-Sobolev inequality, for every $u\in H^{1}(\mathbb{R}^{N})$, the integral
$$
\int_{\mathbb{R}^{N}}\int_{\mathbb{R}^{N}}\frac{|u(x)|^{q}|u(y)|^{q}}{|x-y|^{\mu}}dxdy
$$
is well defined if
$$
\frac{2N-\mu}{N}\leq q\leq\frac{2N-\mu}{N-2}.
$$
Here $\frac{2N-\mu}{N}$ is the lower critical exponent and $2_{\mu}^{\ast}:=\frac{2N-\mu}{N-2}$ is the upper critical exponent due to the Hardy-Littlewood-Sobolev inequality.
The critical problem for the Choquard equation is an interesting topic and has attracted a lot of attention recently. The lower critical exponent case was studied in \cite{MS2}, some existence and nonexistence results were established if the potential $1-V$ should not decay to zero at infinity faster than the inverse square of $|x|$.
In order to study the critical nonlocal equation with upper critical exponent $2_{\mu}^{\ast}$, we use $S_{H,L}$ to denote the best constant defined by
\begin{equation}\label{S1}
S_{H,L}:=\displaystyle\inf\limits_{u\in D^{1,2}(\mathbb{R}^N)\backslash\{{0}\}}\ \ \frac{\displaystyle\int_{\mathbb{R}^N}|\nabla u|^{2}dx}{\Big(\displaystyle\int_{\mathbb{R}^N}\int_{\mathbb{R}^N}
\frac{|u(x)|^{2_{\mu}^{\ast}}|u(y)|^{2_{\mu}^{\ast}}}{|x-y|^{\mu}}dxdy\Big)^{\frac{N-2}{2N-\mu}}}.
\end{equation}
A critical Choquard type equation on a bounded domain of $\R^N$, $N\geq 3$  was investigated in \cite{GY, GY2}, there the authors generalized the well-known results obtained in \cite{ABC, BN}. In \cite{GY} it was observed that
\begin{Prop}\label{ExFu} (See \cite{GY}.)
The constant $S_{H,L}$ defined in \eqref{S1} is achieved if and only if $$u=C\left(\frac{b}{b^{2}+|x-a|^{2}}\right)^{\frac{N-2}{2}} ,$$ where $C>0$ is a fixed constant, $a\in \mathbb{R}^{N}$ and $b\in(0,\infty)$ are parameters. What's more,
$$
S_{H,L}=\frac{S}{C(N,\mu)^{\frac{N-2}{2N-\mu}}},
$$
where $S$ is the best Sobolev constant.
\end{Prop}
Let $U(x):=\frac{[N(N-2)]^{\frac{N-2}{4}}}{(1+|x|^{2})^{\frac{N-2}{2}}}$ be a minimizer for $S$, see \cite{Wi} for example, then
\begin{equation}\label{REL}
\aligned
\tilde{U}(x)=S^{\frac{(N-\mu)(2-N)}{4(N-\mu+2)}}C(N,\mu)^{\frac{2-N}{2(N-\mu+2)}}U(x)
\endaligned
\end{equation}
is the unique  minimizer for $S_{H,L}$ that satisfies
$$
-\Delta u=\Big(\int_{\R^N}\frac{|u(y)|^{2_{\mu}^{\ast}}}{|x-y|^{\mu}}dy\Big)|u|^{2_{\mu}^{\ast}-2}u\ \ \   \hbox{in}\ \ \ \R^N
$$
and
$$
\int_{\mathbb{R}^N}|\nabla \tilde{U}|^{2}dx=\int_{\mathbb{R}^N}\int_{\mathbb{R}^N}\frac{|\tilde{U}(x)|^{2_{\mu}^{\ast}}|\tilde{U}(y)|^{2_{\mu}^{\ast}}}{|x-y|^{\mu}}dxdy=S_{H,L}^{\frac{2N-\mu}{N-\mu+2}}.
$$
Moreover, for every open subset $\Omega$ of $\mathbb{R}^N$,
\begin{equation}
S_{H,L}(\Omega):=\displaystyle\inf\limits_{u\in D_{0}^{1,2}(\Omega)\backslash\{{0}\}}\ \ \frac{\displaystyle\int_{\Omega}|\nabla u|^{2}dx}{\left(\displaystyle\int_{\Omega}\int_{\Omega}\frac{|u(x)|^{2_{\mu}^{\ast}}|u(y)|^{2_{\mu}^{\ast}}}{|x-y|^{\mu}}dxdy\right)^{\frac{N-2}{2N-\mu}}}=S_{H,L},
\end{equation}
$S_{H,L}(\Omega)$ is never achieved except when $\Omega=\R^N$. That means, for bounded domain $\Omega$ there are no nontrivial solutions for $$
-\Delta u=\Big(\int_{\Omega}\frac{|u(y)|^{2_{\mu}^{\ast}}}{|x-y|^{\mu}}dy\Big)|u|^{2_{\mu}^{\ast}-2}u\ \ \   \hbox{in}\ \ \ \Omega.
$$
On the other hand, similar to the observation made in \cite{BC}, if $V(x)=\lambda$ is a constant and $q=\frac{2N-\mu}{N-2}$ in \eqref{Nonlocal.S1} while $u$ is a classical solution, then we can establish the following Poho\u{z}aev identity
$$
\frac{N-2}{2}\int_{\R^N} |\nabla u|^{2}dx+\frac{\lambda N}{2}\int_{\R^N} |u|^{2}dx=\frac{N-2}{2}\int_{\R^N}
\int_{\R^N}\frac{|u(x)|^{2_{\mu}^{\ast}}|u(y)|^{2_{\mu}^{\ast}}}{|x-y|^{\mu}}dxdy,
$$
thus we can obtain
$$
\lambda\int_{\R^N} |u|^{2}dx=0,
$$
which means that there are no nontrivial solutions with $\lambda\neq 0$. Hence it is quite interesting to know how the behavior of the potential function or the perturbation of the critical term will affect the existence of solutions for critical Choquard equation.

If the critical part was perturbed by a subcritical term, the existence of ground states was investigated in \cite{AGSY} there the authors also studied the semiclassical limit problem for the singularly perturbed Choquard equation in $\R^3$ and characterized
the concentration behavior by variational methods. For the problem with sign-changing potential, a strongly indefinite Choquard equation with critical exponent was studied in \cite{GY3} via generalized linking theorem.
Recently the case of critical growth in the sense of Trudinger-Moser inequality in $\R^2$ was also considered in \cite{ACTY}, there the authors studied the existence and concentration of the ground states.

The aim of the present paper is to consider the nonlinear Choquard equation with potential well, that is
\begin{equation}\label{CCE}
\left\{\begin{array}{l}
\displaystyle-\Delta u+(\lambda V(x)-\beta)u
=\big(|x|^{-\mu}\ast |u|^{2_{\mu}^{\ast}}\big)|u|^{2_{\mu}^{\ast}-2}u\hspace{4.14mm}\mbox{in}\hspace{1.14mm} \mathbb{R}^N,\\
\displaystyle u\in H^{1}(\mathbb{R}^N),\hspace{10.6mm}
\end{array}
\right.
\end{equation}
where $\lambda,\beta\in\mathbb{R}^+$, $0<\mu<N$, $N\geq4$ and the potential $V$ satisfies the assumptions:

 $(V_1)$ $V\in \mathcal{C}(\mathbb{R}^N,\mathbb{R})$, $V\geq0$, and $\Omega :=$ int $V^{-1}(0)$ is a nonempty bounded set with smooth boundary, 0 is in interior of $\Omega$ and $\overline{\Omega}=V^{-1}(0)$.

$(V_2)$ There exists $M_{0}>0$ such that
$$
\mathcal{L}\{x\in\mathbb{R}^N:V(x)\leq M_{0}\}<\infty,
$$
where $\mathcal{L}$ denotes the Lebesgue measure in $\mathbb{R}^N$.

As we all know, the local nonlinear Schr\"{o}dinger equation with deepening potential well has also been widely investigated. Consider
\begin{equation}\label{local.S1}
-\Delta u+(\lambda V(x)-\beta)u=|u|^{p-2}u,\hspace{4.14mm} \mbox{in}\ \mathbb{R}^N,
\end{equation}
where the potential $V(x)$ satisfies $(V_1)$ and $(V_2)$. In \cite{BW}, the authors studied the subcritical case and proved the existence of a least energy solution of \eqref{local.S1} for large $\lambda$. They also showed that the sequence of least energy solutions converges strongly to a least energy solution for a problem in bounded domain. Furthermore, they also obtained the existence of at least $cat(\Omega)$ positive solutions for large $\lambda$, where $\Omega =\mbox{int} (V^{-1}(0))$ and $cat(\Omega)$ stands for the category of the domain  $\Omega$.

The critical case was considered
in \cite{CD1}, there the authors proved the existence and multiplicity of positive solutions which localize near the potential well for $\beta$ small and $\lambda$ large. Later, they also proved the existence of solutions which change sign exactly once in \cite{CD2}. We also refer to \cite{BW3} where the authors proved the existence of $k$ solutions that may change sign for any $k$ and $\lambda$ large enough. Suppose that the potential $V(x)$ satisfies \eqref{a0}, \eqref{a00} and the nonlinearity is of subcritical growth, the authors in \cite{DT} overcame the loss of compactness and  applied the deformation flow arguments to build the multi-bump shaped solutions. Recently the existence of multi-bump shaped solutions for \eqref{local.S1} with critical growth was also studied in \cite{GT, GT2, Tz}, the main results there generalize and complement the theorems in \cite{DT}. We would also like to mention some related nonlocal problems in \cite{JZ} and the references therein, there the existence of solutions of the nonlocal Schr\"odinger-Poisson system was investigated under the effect of critical growth assumption or potential well type function $V(x)$. It is then quite natural to ask how the appearance of the potential well will affect the existence of solutions of the critical Choquard equation  \eqref{CCE} and what is the asymptotic behavior of the solutions as the parameter $\lambda$ goes to infinity, does the same results established for local Schr\"odinger equation still hold for the critical Choquard equation?

To study equation \eqref{CCE} by variational methods, we introduce the energy functional defined by
$$
J_{\lambda,\beta}(u)=\frac{1}{2}\int_{\mathbb{R}^N}(|\nabla u|^{2}+(\lambda V(x)-\beta)|u|^{2})dx-\frac{1}{2\cdot2_{\mu}^{\ast}}\int_{\mathbb{R}^N}\int_{\mathbb{R}^N}
\frac{|u(x)|^{2_{\mu}^{\ast}}|u(y)|^{2_{\mu}^{\ast}}}{|x-y|^{\mu}}dxdy.
$$
The Hardy-Littlewood-Sobolev inequality implies that $J_{\lambda,\beta}$ is well defined on $H^{1}(\mathbb{R}^N)$ and belongs to $\mathcal{C}^{1}$. Then we see that $u$ is a weak solution of \eqref{CCE} if and only if $u$ is a critical point of the functional $J_{\lambda,\beta}$. Furthermore, a function $u_{0}$ is called a ground state of \eqref{CCE} if $u_{0}$ is a critical point of \eqref{CCE} and $J_{\lambda,\beta}(u_{0})\leq J_{\lambda,\beta}(u)$ holds for any critical point $u$ of \eqref{CCE}, i.e.
$$
J_{\lambda,\beta}(u_{0})=c:=\inf \Big\{J_{\lambda,\beta}(u):u\in H^{1}(\mathbb{R}^N)\backslash\{{0}\} \mbox{ is a critical point of \eqref{CCE}} \Big\}.
$$

In the following we will denote the sequence of eigenvalues of the operator $-\Delta$ on $\Omega$ with homogeneous Dirichlet boundary data by
$$
0<\beta_{1}<\beta_{2}\leq...\leq \beta_{j}\leq\beta_{j+1}\leq...
$$
and
$
\beta_{j}\rightarrow+\infty
$
as $j\rightarrow+\infty$. Notice that wether the parameter $\beta$ lies in $(0, \beta_{1})$ or not affect the functional  $J_{\lambda,\beta}$ greatly. If $0<\beta<\beta_{1}$, the operator $-\Delta+\lambda V(x)-\beta$ is positively definite in $H^{1}(\mathbb{R}^N)$. However, if $\beta>\beta_{1}$, the operator $-\Delta+\lambda V(x)-\beta$ might be indefinite in $H^{1}(\mathbb{R}^N)$. Moreover, the appearance of convolution type nonlinearities brings us a lot of difficulties and the techniques in \cite{BW, GT2, Tz} can not be applied to the Choquard equation directly.  Thus, to look for solutions for equation \eqref{CCE}, we need to develop new techniques to overcome the difficulties.

The first result is to establish the existence of ground state solutions and the asymptotic behavior of the solutions for \eqref{CCE} with $\beta\in (0,\beta_{1})$. The result reads as
\begin{thm}\label{EXS}
Suppose that assumptions $(V_1)$ and $(V_2)$ hold, $0<\mu<N$, $N\geq4$. Then, for any $\beta\in (0,\beta_{1})$ there exists $\lambda_{\beta}>0$ such that, for each $\lambda\geq\lambda_{\beta}$, equation \eqref{CCE} has at least one ground state solution $u$, where $\beta_{1}$ is the first eigenvalue of $-\Delta$ on $\Omega$ with boundary condition $u=0$. Furthermore, for any sequences $\lambda_{n}\rightarrow\infty$, then every sequence of solutions $\{u_{n}\}$ of \eqref{CCE} satisfying $J_{\lambda,\beta}(u_{n})\rightarrow c<\frac{N+2-\mu}{4N-2\mu}S_{H,L}^{\frac{2N-\mu}{N+2-\mu}}$ as $n\rightarrow\infty$, converges to a solution of
\begin{equation}\label{LCCE}
\left\{\begin{array}{l}
\displaystyle-\Delta u
=\beta u+\big(|x|^{-\mu}\ast |u|^{2_{\mu}^{\ast}}\big)|u|^{2_{\mu}^{\ast}-2}u\hspace{6.14mm} \mbox{in}\hspace{1.14mm} \Omega,\\
\displaystyle u\in H_{0}^{1}(\Omega),
\end{array}
\right.
\end{equation}
$\Omega$ is defined as in $(V_{1})$.
\end{thm}

Next we will use the Lusternik-Schnirelmann category (see e.g. \cite{Wi}) to characterize the multiplicity result.
\begin{thm}\label{EXS3}
Assume $(V_1)$ and $(V_2)$ hold, $0<\mu<N$ and $N\geq4$. Then, there exist $0<\beta^{*}<\beta_{1}$ and for each $0<\beta\leq\beta^{*}$ two numbers $\lambda_{\beta}>0$ and $0<c_{\beta}<\frac{N+2-\mu}{4N-2\mu}S_{H,L}^{\frac{2N-\mu}{N+2-\mu}}$ such that, if $\lambda\geq\lambda_{\beta}$, then \eqref{CCE} has at least $cat(\Omega)$ solutions with energy $J_{\lambda,\beta}\leq c_{\beta}$,  where $cat(\Omega)$ is the category of the domain  $\Omega$.
\end{thm}

Finally we are interested in the critical Choquard equation \eqref{CCE} with indefinite potential. In this case we  assume that $\beta>\beta_{1},\beta\neq\beta_{j}$ for any $j>1$ and introduce assumption

$(V_3)$ $\liminf_{|x|\rightarrow\infty}V(x)>0$.

 The result says that
\begin{thm}\label{EXS4}
Suppose that assumptions $(V_1)$ and $(V_3)$ hold, $0<\mu<4$, $N\geq4$. Then, for any $\beta>\beta_{1},\beta\neq\beta_{j}$, $j>1$,  there exists $\lambda_{\beta}>0$ such that, for each $\lambda\geq\lambda_{\beta}$, equation \eqref{CCE} has at least one ground state solution $u_{\lambda}$. Furthermore, for any sequences $\lambda_{n}\rightarrow\infty$, the solution sequence $\{u_{\lambda_{n}}\}$ has a subsequence converging to a ground state solution  $u$ of \eqref{LCCE}.
\end{thm}

\begin{Rem}\label{E0}
Obviously assumption $(V_3)$ is stronger than
assumption $(V_2)$. To see this, we only need to take $M_{0}=\frac{1}{2}\liminf_{|x|\rightarrow\infty}V(x)$.
\end{Rem}

Throughout this paper we write $|\cdot|_{q}$ for the $L^{q}(\mathbb{R}^N)$-norm, $q\in[1,\infty]$ and always assume that conditions $(V_1)$ and $(V_2)$ hold in Sections 2-4, conditions $(V_1)$ and $(V_3)$ hold in Sections 5-6, $0<\mu<N$ and $N\geq4$. We denote by $C, C_{1}, C_{2}, C_{3}, \cdots$ the different positive constants and
$$
\|u\|_{H^{1}}^{2}:=\int_{\mathbb{R}^N}(|\nabla u|^{2}+u^{2})dx
$$
the standard norm on $H^{1}(\mathbb{R}^N)$.

An outline of the paper is as follows: In Section 2, we give some preliminary results for the case $0<\beta<\beta_{1}$ and prove Palais-Smale condition ($(PS)$ condition, for short). In Section 3, we prove the existence of ground states for \eqref{CCE} by a problem on bounded region and show the certain concentration behavior of the solutions occurs as $\lambda\rightarrow\infty$. In Section 4, the Lusternik-Schnirelmann theory would give the existence of at least $cat(\Omega)$ critical points for \eqref{CCE}. In Section 5, we give some preliminary results for the case $\beta>\beta_{1}$, $\beta\neq\beta_{j}$ for any $j>1$. In Section 6, we prove the existence of ground states for \eqref{CCE} with indefinite potential and show the certain concentration behavior of the solutions occurs as $\lambda\rightarrow\infty$.

\section{Existence of solutions for the case $0<\beta<\beta_{1}$}
Next we denote by
$$
E=\left\{u\in H^{1}(\mathbb{R}^N):\int_{\mathbb{R}^{N}}Vu^{2}dx<+\infty\right\}
$$
the Hilbert space equipped with norm
$$
\|u\|=\left(\|u\|_{H^{1}}^{2}+\int_{\mathbb{R}^{N}}Vu^{2}dx\right)^{\frac{1}{2}}.
$$
If $\lambda>0$, then it is equivalent to the norms
$$
\|u\|_{\lambda}=\left(\|u\|_{H^{1}}^{2}+\lambda\int_{\mathbb{R}^{N}}Vu^{2}dx\right)^{\frac{1}{2}}.
$$
Obviously, $H_{0}^{1}(\Omega)\subset E$, where $\Omega$ is defined as in $(V_{1})$.

We denote the operator $L_{\lambda,\beta}:=-\Delta+\lambda V(x)-\beta$ and particularly, $L_{\lambda,0}:=-\Delta+\lambda V(x)$ and $L_{0,\beta}:=-\Delta-\beta$. Observe that
$$
0\leq a_{\lambda}=\inf \{\langle L_{\lambda,0}u,u\rangle:u\in E,|u|_{2}=1\}
$$
and that $a_{\lambda}$ is nondecreasing in $\lambda$.

The following two Lemmas are taken from \cite{CD1}.
\begin{lem}\label{EMB}
If $u_{n}\in E$ be such that $\lambda_{n}\rightarrow\infty$ and $\|u_{n}\|_{\lambda_{n}}^{2}<C$. Then, there is a $u\in H_{0}^{1}(\Omega)$ such that, up to a subsequence, $u_{n}\rightarrow u$ in $L^{2}(\mathbb{R}^N)$.
\end{lem}

\begin{lem}\label{PR2} For every $0<\beta<\beta_{1}$, there exists $\lambda_{\beta}>0$ such that $a_{\lambda}\geq(\beta+\beta_{1})/2$ for $\lambda\geq\lambda_{\beta}$. Consequently,
$$
C_{\beta}\|u\|_{\lambda}^{2}\leq\langle L_{\lambda,\beta}u,u\rangle
$$
for all $u\in E,$ $\lambda\geq\lambda_{\beta}$, where $C_{\beta}>0$ is a constant.
\end{lem}

It follows from Lemma \ref{PR2} that the operator $L_{\lambda,\beta}$ is positive if $\lambda\geq\lambda_{\beta}$ and thus we can introduce on $E$ a new inner product
$$
( u_{1},u_{2})=\langle L_{\lambda,\beta}^{\frac{1}{2}}u_{1},
L_{\lambda,\beta}^{\frac{1}{2}}u_{2}\rangle
$$
 with the norm
$$
\|u\|_{L_{\lambda,\beta}}=( u,u)^{\frac{1}{2}}.
$$
Moreover, noting that for $\beta>0$, 
$$
\|u\|_{L_{\lambda,\beta}}\leq\|u\|_{\lambda}, \ \forall u\in E,
$$
we know $\|u\|_{L_{\lambda,\beta}}$ in fact is equivalent to the norm $\|u\|_{\lambda}$ on $E$ if $\lambda\geq\lambda_{\beta}$. For future use, enlarging $\lambda_{\beta}$ if necessary, we may assume that $\lambda_{\beta}\geq\beta/M_{0}$, thus
\begin{equation}\label{B1}
\lambda M_{0}-\beta\geq0 \hspace{4.14mm}\mbox{for} \hspace{1.14mm}\mbox{all} \hspace{1.14mm}\lambda\geq\lambda_{\beta},
\end{equation}
where $M_{0}$ is given in $(V_2)$.

Since we are considering the critical case, we need to show where the compactness condition is recovered.
\begin{Prop}\label{PS} For each $0<\beta<\beta_{1}$ and $\lambda\geq\lambda_{\beta}$, $J_{\lambda,\beta}$ satisfies the $(PS)_{c}$
condition for all $c<\frac{N+2-\mu}{4N-2\mu}S_{H,L}^{\frac{2N-\mu}{N+2-\mu}}$.
\end{Prop}
\begin{proof}
Let $\{u_{j}\}$ be a $(PS)_{c}$ sequence, i.e.
\begin{equation}\label{B2}
J_{\lambda,\beta}(u_{j})\rightarrow c
\end{equation}
and
\begin{equation}\label{B3}
\sup\{|\langle J_{\lambda,\beta}'(u_{j}),\varphi\rangle|:\varphi\in E,\|\varphi\|_{L_{\lambda,\beta}}=1\}\rightarrow 0
\end{equation}
as $j\rightarrow+\infty$. By \eqref{B2} and \eqref{B3}, for any $j\in\mathbb{N}$, it easily follows that there exists $C_{1}>0$ such that
\begin{equation}\label{B4}
|J_{\lambda,\beta}(u_{j})|\leq C_{1}
\end{equation}
and
\begin{equation}\label{B5}
|\langle J_{\lambda,\beta}'(u_{j}),\frac{u_{j}}{\|u_{j}\|_{L_{\lambda,\beta}}}\rangle|\leq C_{1}.
\end{equation}
Consequently, we have
\begin{equation}\label{B6}
\aligned
\frac{N+2-\mu}{4N-2\mu}\langle L_{\lambda,\beta}u_{j},u_{j}\rangle&=J_{\lambda,\beta}(u_{j})-\frac{1}{2\cdot2_{\mu}^{\ast}}\langle J_{\lambda,\beta}'(u_{j}),u_{j}\rangle\\
&\leq C_{1}(1+\|u_{j}\|_{L_{\lambda,\beta}}),
\endaligned
\end{equation}
that is
$$
\|u_{j}\|_{L_{\lambda,\beta}}^{2}\leq C_{2}(1+\|u_{j}\|_{L_{\lambda,\beta}}),
$$
which means $\{u_{j}\}$ is bounded in $E$.

Now, up to a subsequence, still denoted by $\{u_{j}\}$, we may assume that there exists $u_{\infty}\in E$ such that $u_{j}\rightharpoonup u_{\infty}$ in $E$ and
\begin{equation}\label{B11}
u_{j}\rightarrow u_{\infty} \hspace{3.14mm} a.e. \hspace{2.14mm} \mbox{in} \hspace{2.14mm} \mathbb{R}^N
\end{equation}
as $j\rightarrow+\infty$. From the fact that $|u_{j}|^{2_{\mu}^{\star}}$ is bounded in $L^{\frac{2^{\ast}}{2_{\mu}^{\ast}}}(\mathbb{R}^N)$ we have
$$
|u_{j}|^{2_{\mu}^{\ast}}\rightharpoonup |u_{\infty}|^{2_{\mu}^{\ast}} \hspace{3.14mm} \mbox{in} \hspace{3.14mm} L^{\frac{2N}{2N-\mu}}(\mathbb{R}^N)
$$
as $j\rightarrow+\infty$.
By the Hardy-Littlewood-Sobolev inequality,
the Riesz potential defines a linear continuous map from  $L^{\frac{2N}{2N-\mu}}(\mathbb{R}^N)$ to $L^{\frac{2N}{\mu}}(\mathbb{R}^N)$,  we know that
$$
\int_{\mathbb{R}^N}
\frac{|u_{j}(y)|^{2_{\mu}^{\ast}}}{|x-y|^{\mu}}dy\rightharpoonup \int_{\mathbb{R}^N}
\frac{|u_{\infty}(y)|^{2_{\mu}^{\ast}}}{|x-y|^{\mu}}dy \hspace{3.14mm} \mbox{in} \hspace{3.14mm} L^{\frac{2N}{\mu}}(\mathbb{R}^N)
$$
as $j\rightarrow+\infty$. Combining this with the fact that
$$
|u_{j}|^{2_{\mu}^{\ast}-2}u_{j}\rightharpoonup |u_{\infty}|^{2_{\mu}^{\ast}-2}u_{\infty} \hspace{3.14mm} \mbox{in} \hspace{3.14mm} L^{\frac{2N}{N-\mu+2}}(\mathbb{R}^N)
$$
as $j\rightarrow+\infty$, we have
\begin{equation}\label{B12}
\int_{\mathbb{R}^N}
\frac{|u_{j}(y)|^{2_{\mu}^{\ast}}}{|x-y|^{\mu}}dy|u_{j}(x)|^{2_{\mu}^{\ast}-2}u_{j}(x)\rightharpoonup \int_{\mathbb{R}^N}
\frac{|u_{\infty}(y)|^{2_{\mu}^{\ast}}}{|x-y|^{\mu}}dy|u_{\infty}(x)|^{2_{\mu}^{\ast}-2}u_{\infty}(x) \hspace{3.14mm} \mbox{in} \hspace{3.14mm} L^{\frac{2N}{N+2}}(\mathbb{R}^N)
\end{equation}
as $j\rightarrow+\infty$. Since, for any $\varphi\in E$
$
\langle J_{\lambda,\beta}'(u_{j}),\varphi\rangle\rightarrow0,
$
passing to the limit as $j\rightarrow+\infty$ and taking into account \eqref{B12} we get
$$
\int_{\mathbb{R}^{N}}(\nabla u_{\infty}\nabla\varphi +(\lambda V(x)-\beta) u_{\infty}\varphi )dx
=\int_{\mathbb{R}^N}\int_{\mathbb{R}^N}
\frac{|u_{\infty}(x)|^{2_{\mu}^{\ast}}|u_{\infty}(y)|^{2_{\mu}^{\ast}-2}u_{\infty}(y)\varphi(y)}
{|x-y|^{\mu}}dxdy
$$
for any $\varphi\in E$, that means $u_{\infty}$ is a solution of problem \eqref{CCE}.
Moreover, taking $\varphi=u_{\infty}\in E$ as a test function in \eqref{CCE}, we have
$$
\int_{\mathbb{R}^{N}}(|\nabla u_{\infty}|^{2}+(\lambda V(x)-\beta)  u_{\infty}^{2})dx
=\int_{\mathbb{R}^N}\int_{\mathbb{R}^N}
\frac{|u_{\infty}(x)|^{2_{\mu}^{\ast}}|u_{\infty}(y)|^{2_{\mu}^{\ast}}}{|x-y|^{\mu}}dxdy,
$$
thus
$$
J_{\lambda,\beta}(u_{\infty})=\frac{N+2-\mu}{4N-2\mu}\int_{\mathbb{R}^N}\int_{\mathbb{R}^N}
\frac{|u_{\infty}(x)|^{2_{\mu}^{\ast}}|u_{\infty}(y)|^{2_{\mu}^{\ast}}}{|x-y|^{\mu}}dxdy\geq0.
$$

Now, we write $v_{j}:=u_{j}-u_{\infty}$, then, $v_{j}\rightharpoonup0$ in $E$ and $v_{j}\rightarrow 0$ a.e. in $\mathbb{R}^N$. By the Br\'{e}zis-Lieb type splitting result for nonlocal term in \cite{GY} which says
$$
\int_{\mathbb{R}^N}\int_{\mathbb{R}^N}
\frac{|u_{j}(x)|^{2_{\mu}^{\ast}}|u_{j}(y)|^{2_{\mu}^{\ast}}}
{|x-y|^{\mu}}dxdy=\int_{\mathbb{R}^N}\int_{\mathbb{R}^N}
\frac{|v_{j}(x)|^{2_{\mu}^{\ast}}|v_{j}(y)|^{2_{\mu}^{\ast}}}
{|x-y|^{\mu}}dxdy+\int_{\mathbb{R}^N}\int_{\mathbb{R}^N}
\frac{|u_{\infty}(x)|^{2_{\mu}^{\ast}}|u_{\infty}(y)|^{2_{\mu}^{\ast}}}
{|x-y|^{\mu}}dxdy+o_{j}(1)
$$
as $j\rightarrow+\infty$,
we know that
\begin{equation}\label{B13}
\aligned
c&\leftarrow J_{\lambda,\beta}(u_{j})\\
&=\frac{1}{2}\int_{\mathbb{R}^{N}}(|\nabla v_{j}|^{2}+(\lambda V(x)-\beta) v_{j}^{2})dx+\frac{1}{2}\int_{\mathbb{R}^{N}}(|\nabla u_{\infty}|^{2}+(\lambda V(x)-\beta) u_{\infty}^{2})dx\\
&\hspace{0.3cm}-\frac{1}{2\cdot2_{\mu}^{\ast}}\int_{\mathbb{R}^N}\int_{\mathbb{R}^N}
\frac{|v_{j}(x)|^{2_{\mu}^{\ast}}|v_{j}(y)|^{2_{\mu}^{\ast}}}
{|x-y|^{\mu}}dxdy-\frac{1}{2\cdot2_{\mu}^{\ast}}\int_{\mathbb{R}^N}\int_{\mathbb{R}^N}
\frac{|u_{\infty}(x)|^{2_{\mu}^{\ast}}|u_{\infty}(y)|^{2_{\mu}^{\ast}}}
{|x-y|^{\mu}}dxdy+o_{j}(1)\\
&=J_{\lambda,\beta}(u_{\infty})+J_{\lambda,\beta}(v_{j})+o_{j}(1).
\endaligned
\end{equation}
Analogously, we have
$$
\langle J_{\lambda,\beta}'(u_{j}),u_{j}\rangle =\langle J_{\lambda, \beta}'(u_{\infty}),u_{\infty}\rangle+\langle J_{\lambda,\beta}'(v_{j}),v_{j}\rangle+o_{j}(1).
$$
It follows from $\langle J_{\lambda,\beta}'(u_{\infty}),u_{\infty}\rangle=0$ and $\langle J_{\lambda,\beta}'(u_{j}),u_{j}\rangle\rightarrow0$ that
$$
\int_{\mathbb{R}^{N}}(|\nabla v_{j}|^{2}+(\lambda V(x)-\beta) v_{j}^{2})dx\rightarrow b
\ \
\mbox{and}
\ \
\int_{\mathbb{R}^N}\int_{\mathbb{R}^N}
\frac{|v_{j}(x)|^{2_{\mu}^{\ast}}|v_{j}(y)|^{2_{\mu}^{\ast}}}
{|x-y|^{\mu}}dxdy\rightarrow b.
$$
Since $
J_{\lambda,\beta}(u_{\infty})\geq0
$ and \eqref{B13}, we obtain,
\begin{equation}\label{B14}
c\geq \frac{N+2-\mu}{4N-2\mu}b.
\end{equation}
By Lemma \ref{EMB} one knows that as $j\rightarrow\infty$,
$
\displaystyle\int_{F}|v_{j}|^{2}dx\rightarrow0,
$ where $F=\{x\in\mathbb{R}^N:V(x)\leq M_{0}\}$. Let $F^{c}=\mathbb{R}^N\backslash F$. Then, from the definition of $S_{H,L}$ and \eqref{B1}, we have
$$\aligned
S_{H,L}\Big(\int_{\mathbb{R}^N}&\int_{\mathbb{R}^N}
\frac{|v_{j}(x)|^{2_{\mu}^{\ast}}|v_{j}(y)|^{2_{\mu}^{\ast}}}
{|x-y|^{\mu}}dxdy\Big)^{\frac{N-2}{2N-\mu}}\\
&\leq\int_{\mathbb{R}^N}|\nabla v_{j}|^{2}dx\\
&\leq\int_{\mathbb{R}^N}|\nabla v_{j}|^{2}dx+\int_{F^{c}} (\lambda V(x)-\beta)|v_{j}|^{2}dx\\
&\leq\int_{\mathbb{R}^{N}}(|\nabla v_{j}|^{2}+(\lambda V(x)-\beta)|v_{j}|^{2})dx+\beta\int_{F}|v_{j}|^{2}dx\\
&=\int_{\mathbb{R}^{N}}(|\nabla v_{j}|^{2}+(\lambda V(x)-\beta)|v_{j}|^{2})dx+o_{j}(1),
\endaligned
$$
passing to the limit, it yields that $b\geq S_{H,L}b^{\frac{N-2}{2N-\mu}}$. Then we have either $b=0$ or $b\geq S_{H,L}^{\frac{2N-\mu}{N-\mu+2}}$. If $b=0$, the proof is complete. Otherwise $b\geq S_{H,L}^{\frac{2N-\mu}{N-\mu+2}}$, then we can obtain from \eqref{B14},
$$
\frac{N+2-\mu}{4N-2\mu}S_{H,L}^{\frac{2N-\mu}{N-\mu+2}}\leq\frac{N+2-\mu}{4N-2\mu}b\leq c,
$$
which contradicts with the fact that $c<\frac{N+2-\mu}{4N-2\mu}S_{H,L}^{\frac{2N-\mu}{N+2-\mu}}$. Thus $b=0$, and
$$
\|u_{j}-u_{\infty}\|_{L_{\lambda,\beta}}\rightarrow0
$$
as $j\rightarrow+\infty$. This ends the proof of Proposition \ref{PS}.
\end{proof}

\section{Proof of Theorem \ref{EXS}}

It is convenient to show that the functional $J_{\lambda,\beta}$ satisfies the Mountain-Pass geometry.

\begin{lem}\label{MPE} For any $0<\beta<\beta_{1}$, $\lambda>0$ large enough, the functional $J_{\lambda,\beta}$ satisfies the following conditions.\\
(i) There exist $\alpha,\rho>0$ such that $J_{\lambda,\beta}(u)\geq\alpha$ for $\|u\|_{L_{\lambda,\beta}}=\rho$. \\
(ii) There exists a $w_{1}\in E$ with $\|w_{1}\|_{L_{\lambda,\beta}}>\rho$ such that $J_{\lambda,\beta}(w_{1})<0$.
\end{lem}
\begin{proof} (i) By $0<\beta<\beta_{1}$, the Sobolev embedding and Hardy-Littlewood-Sobolev inequality, for all $u\in E\backslash\ \{0\}$ we have
$$
\aligned
J_{\lambda,\beta}(u)&\geq\frac{1}{2}\int_{\mathbb{R}^{N}}(|\nabla u|^{2}+(\lambda V(x)-\beta)|u|^{2})dx-\frac{1}{2\cdot2_{\mu}^{\ast}}C_{1}|u|_{2^{\ast}}^{2\cdot2_{\mu}^{\ast}}\\
&\geq C_{2}\|u\|_{L_{\lambda,\beta}}^{2}-C_{3}\|u\|_{L_{\lambda,\beta}}^{2\cdot2_{\mu}^{\ast}}.\\
\endaligned
$$
Since $2<2\cdot2_{\mu}^{\ast}$, we can choose some $\alpha,\rho>0$ such that $J_{\lambda,\beta}(u)\geq\alpha$ for $\|u\|_{L_{\lambda,\beta}}=\rho$.

(ii) For any $u_{1}\in E\backslash\ \{0\}$, we have
$$
J_{\lambda,\beta}(tu_{1})=\frac{t^{2}}{2}\int_{\mathbb{R}^{N}}(|\nabla u_{1}|^{2}+(\lambda V(x)-\beta)u_{1}^{2})dx-\frac{t^{2\cdot2_{\mu}^{\ast}}}{2\cdot2_{\mu}^{\ast}}
\int_{\mathbb{R}^N}\int_{\mathbb{R}^N}
\frac{|u_{1}(x)|^{2_{\mu}^{\ast}}|u_{1}(y)|^{2_{\mu}^{\ast}}}
{|x-y|^{\mu}}dxdy<0
$$
for $t>0$ large enough. Hence, we can take a $w_{1}:=t_{1}u_{1}$ for some $t_{1}>0$ and (ii) follows.
\end{proof}
Applying the mountain pass theorem without $(PS)$ condition (cf. \cite{Wi}), there
exists a $(PS)$ sequence $\{u_{n}\}$ such that $J_{\lambda,\beta}(u_{n})\rightarrow c$ and $ J_{\lambda,\beta}'(u_{n})\rightarrow0$ in $E^{-1}$ at the minimax level
$$
c_{\lambda,\beta}=\inf\limits_{\gamma\in\Gamma}\max\limits_{t\in[0,1]}J_{\lambda,\beta}(\gamma(t))>0,
$$
where
$$
\Gamma:=\{\gamma\in C([0,1],E):\gamma(0)=0,J_{\lambda,\beta}(\gamma(1))<0\}.
$$
If we denote the Nehari manifold of $J_{\lambda,\beta}$ by
$$
{\cal M}_{\lambda,\beta}=\{u\in E\backslash\{{0}\}:\langle J_{\lambda,\beta}'(u),u\rangle=0\},
$$
since $0<\beta<\beta_{1}$ and $2<2\cdot2_{\mu}^{\ast}$, the function $t\in\mathbb{R_{+}}\rightarrow J_{\lambda,\beta}(tu)$ has an unique maximum point $t(u)>0$ and $t(u)u\in {\cal M}_{\lambda,\beta}$. Then $c_{\lambda,\beta}$ has an equivalent minimax characterization, that is
\begin{equation}\label{C1}
c_{\lambda,\beta}:=\inf\limits_{u\in {\cal M}_{\lambda,\beta}}J_{\lambda,\beta}(u)=\inf\limits_{u\in E,u\neq0}\max\limits_{t\geq0}J_{\lambda,\beta}(tu).
\end{equation}

Next we denote by $J_{\beta,\Omega}$ the restriction of $J_{\lambda,\beta}$ on $H_{0}^{1}(\Omega)$, that is
$$
J_{\beta,\Omega}(u)=\frac{1}{2}\int_{\Omega}|\nabla u|^{2}dx-\frac{\beta}{2}\int_{\Omega} |u|^{2}dx-\frac{1}{2\cdot2_{\mu}^{\ast}}\int_{\Omega}
\int_{\Omega}\frac{|u(x)|^{2_{\mu}^{\ast}}|u(y)|^{2_{\mu}^{\ast}}}
{|x-y|^{\mu}}dxdy,
$$
where $\Omega$ is defined as in $(V_{1})$. The Nehari manifold of $J_{\beta,\Omega}$ is
$$
{\cal M}_{\beta,\Omega}=\{u\in H_{0}^{1}(\Omega)\backslash\{{0}\}:\langle J_{\beta,\Omega}'(u),u\rangle=0\}.
$$
Set
$$
c(\beta,\Omega):=\inf\limits_{u\in {\cal M}_{\beta,\Omega}}J_{\beta,\Omega}(u).
$$
Analogously, we have
\begin{equation}\label{C3}
c(\beta,\Omega)=\inf\limits_{u\in H_{0}^{1}(\Omega)),u\neq0}\max\limits_{t\geq0}J_{\beta,\Omega}(tu)=\inf\limits_{\gamma\in\Gamma}\max\limits_{t\in[0,1]}J_{\beta,\Omega}(\gamma(t)),
\end{equation}
where
$$
\Gamma:=\{\gamma\in C([0,1],H_{0}^{1}(\Omega)):\gamma(0)=0,J_{\beta,\Omega}(\gamma(1))<0\}.
$$

The following Lemma will plays an important role in estimating the Mountain pass levels. By the proof of Theorem 1.4 (i) in \cite{GY}, we have
\begin{lem}\label{MPE3} Let $\beta>0$, $\beta\neq\beta_{j}$ for any $j\geq1$. There exists $e\in H_{0}^{1}(\Omega)\backslash\{{0}\}$ such that
\begin{equation}\label{C4}
\sup_{t\geq0}J_{\beta,\Omega}(te)<\frac{N+2-\mu}{4N-2\mu}S_{H,L}^{\frac{2N-\mu}{N+2-\mu}}.
\end{equation}
\end{lem}

\begin{Prop}\label{MPE4} Let $\beta>0$, $\beta\neq\beta_{j}$ for any $j\geq1$. If $\lambda\geq\lambda_{\beta}$ then
\begin{equation}\label{C7}
0< c_{\lambda,\beta}\leq c(\beta,\Omega)<\frac{N+2-\mu}{4N-2\mu}S_{H,L}^{\frac{2N-\mu}{N+2-\mu}}.
\end{equation}
\end{Prop}

\begin{proof} Lemma \ref{MPE} implies $c_{\lambda,\beta}>0$. Since
$$
\left\{u\in H_{0}^{1}(\Omega):\displaystyle\int_{\mathbb{R}^N}\int_{\mathbb{R}^N}
\frac{|u(x)|^{2_{\mu}^{\ast}}|u(y)|^{2_{\mu}^{\ast}}}
{|x-y|^{\mu}}dxdy=1\right\}\subset \left\{u\in E:\displaystyle\int_{\mathbb{R}^N}\int_{\mathbb{R}^N}
\frac{|u(x)|^{2_{\mu}^{\ast}}|u(y)|^{2_{\mu}^{\ast}}}
{|x-y|^{\mu}}dxdy=1\right\}
$$
and $\langle L_{\lambda,\beta}u,u\rangle=\langle L_{0,\beta}u,u\rangle$ for
$$
u\in \left\{u\in H_{0}^{1}(\Omega):\displaystyle\int_{\mathbb{R}^N}\int_{\mathbb{R}^N}
\frac{|u(x)|^{2_{\mu}^{\ast}}|u(y)|^{2_{\mu}^{\ast}}}
{|x-y|^{\mu}}dxdy=1\right\},$$
it follows that $c_{\lambda,\beta}\leq c(\beta,\Omega)$. By Lemma \ref{MPE3} and \eqref{C3}, we know $c(\beta,\Omega)<\frac{N+2-\mu}{4N-2\mu}S_{H,L}^{\frac{2N-\mu}{N+2-\mu}}$.
Hence, the conclusion is proved.
\end{proof}

\noindent
{\bf Proof of Theorem \ref{EXS}.}
Applying the Mountain-Pass theorem without $(PS)$ condition, we know there exists a $(PS)_{c_{\lambda,\beta}}$ sequence $\{u_{n}\}$. Then we obtain from Proposition \ref{PS} and Proposition \ref{MPE4}, \eqref{CCE} has at least one ground state solution $u$.

In the following, we come to give the asymptotic behavior of the solutions of \eqref{CCE} as $\lambda$ goes to infinity.
For $0<\beta<\beta_{1}$, let $\{u_{n}\}$ be a sequence of solutions of \eqref{CCE} such that $\lambda_{n}\rightarrow\infty$ and $J_{\lambda_{n},\beta}(u_{n})\rightarrow c<\frac{N+2-\mu}{4N-2\mu}S_{H,L}^{\frac{2N-\mu}{N+2-\mu}}$, we have
$$\aligned
J_{\lambda_{n},\beta}(u_{n})
&=\frac{N+2-\mu}{4N-2\mu}\int_{\mathbb{R}^N}\int_{\mathbb{R}^N}
\frac{|u_{n}(x)|^{2_{\mu}^{\ast}}|u_{n}(y)|^{2_{\mu}^{\ast}}}
{|x-y|^{\mu}}dxdy,\\
\endaligned
$$
and so,
\begin{equation}\label{D1}
\lim\limits_{n\rightarrow\infty}\int_{\mathbb{R}^N}\int_{\mathbb{R}^N}
\frac{|u_{n}(x)|^{2_{\mu}^{\ast}}|u_{n}(y)|^{2_{\mu}^{\ast}}}
{|x-y|^{\mu}}dxdy<S_{H,L}^{\frac{2N-\mu}{N+2-\mu}}.
\end{equation}
By Lemma \ref{PR2}, we can deduce that
$$
\frac{N+2-\mu}{4N-2\mu}S_{H,L}^{\frac{2N-\mu}{N+2-\mu}}>J_{\lambda_{n},\beta}(u_{n})=\frac{N+2-\mu}{4N-2\mu}\langle L_{\lambda_{n},\beta}u_{n},u_{n}\rangle\geq \frac{N+2-\mu}{4N-2\mu}C_{\beta}\|u_{n}\|_{L_{\lambda_{n},\beta}}^{2},
$$
and so $\{u_{n}\}$ is bounded in $E$.

By Lemma \ref{EMB}, there is a $u\in H_{0}^{1}(\Omega)$ such that, up to a subsequence, $u_{n}\rightharpoonup u$ in $E$ and $u_{n}\rightarrow u$ in $L^{2}(\mathbb{R}^N)$. From the fact that $u_{n}$ is a solution of \eqref{CCE}, we have
$$
\int_{\mathbb{R}^{N}}(\nabla u_{n}\nabla\varphi+(\lambda_{n} V-\beta)u_{n}\varphi )dx
=\int_{\mathbb{R}^N}\int_{\mathbb{R}^N}
\frac{|u_{n}(x)|^{2_{\mu}^{\ast}}|u_{n}(y)|^{2_{\mu}^{\ast}-2}u_{n}(y)\varphi(y)}
{|x-y|^{\mu}}dxdy
$$
for any $\varphi\in E$. If $\varphi\in H_{0}^{1}(\Omega)$ then $\lambda_{n}\displaystyle\int_{\mathbb{R}^N} Vu_{n}\varphi dx=0$ for all $n$. Letting $n\rightarrow\infty$ we obtain
$$
\int_{\mathbb{R}^{N}}(\nabla u\nabla\varphi
-\beta u\varphi) dx=\int_{\mathbb{R}^N}\int_{\mathbb{R}^N}
\frac{|u(x)|^{2_{\mu}^{\ast}}|u(y)|^{2_{\mu}^{\ast}-2}u(y)\varphi(y)}
{|x-y|^{\mu}}dxdy
$$
for any $\varphi\in H_{0}^{1}(\Omega)$. So, $u$ is a solution of \eqref{LCCE}. Define $v_{n}:=u_{n}-u$, then $v_{n}\rightarrow0$ in $L^{2}(\mathbb{R}^N)$ and $v_{n}\rightarrow 0$ a.e. in $\mathbb{R}^N$ as $n\rightarrow+\infty$.

Since $V(x)=0$ for $x\in\Omega$, we get
\begin{equation}\label{D2}
\langle L_{\lambda_{n},\beta}u_{n},u_{n}\rangle=\langle L_{0,\beta} u,u\rangle+\langle L_{\lambda_{n},\beta}v_{n},v_{n}\rangle.
\end{equation}
Since $\{u_{n}\}$ is a sequence of solutions of \eqref{CCE} and $u$ is a solution of \eqref{LCCE}, by the Br\'{e}zis-Lieb type splitting result for nonlocal term in \cite{GY} that
$$
\int_{\mathbb{R}^N}\int_{\mathbb{R}^N}
\frac{|u_{n}(x)|^{2_{\mu}^{\ast}}|u_{n}(y)|^{2_{\mu}^{\ast}}}
{|x-y|^{\mu}}dxdy=\int_{\mathbb{R}^N}\int_{\mathbb{R}^N}
\frac{|u(x)|^{2_{\mu}^{\ast}}|u(y)|^{2_{\mu}^{\ast}}}
{|x-y|^{\mu}}dxdy+\int_{\mathbb{R}^N}\int_{\mathbb{R}^N}
\frac{|v_{n}(x)|^{2_{\mu}^{\ast}}|v_{n}(y)|^{2_{\mu}^{\ast}}}
{|x-y|^{\mu}}dxdy+o_{n}(1),
$$
we can get
\begin{equation}\label{D3}
\langle L_{\lambda_{n},\beta}v_{n},v_{n}\rangle-\int_{\mathbb{R}^N}\int_{\mathbb{R}^N}
\frac{|v_{n}(x)|^{2_{\mu}^{\ast}}|v_{n}(y)|^{2_{\mu}^{\ast}}}
{|x-y|^{\mu}}dxdy=o_{n}(1).
\end{equation}
We claim that $$\displaystyle\int_{\mathbb{R}^N}\int_{\mathbb{R}^N}
\frac{|v_{n}(x)|^{2_{\mu}^{\ast}}|v_{n}(y)|^{2_{\mu}^{\ast}}}
{|x-y|^{\mu}}dxdy\rightarrow0.$$
Assume by contrary that
$$\displaystyle\int_{\mathbb{R}^N}\int_{\mathbb{R}^N}
\frac{|v_{n}(x)|^{2_{\mu}^{\ast}}|v_{n}(y)|^{2_{\mu}^{\ast}}}
{|x-y|^{\mu}}dxdy\rightarrow b>0.$$
Then,
$$\aligned
S_{H,L}&\left(\int_{\mathbb{R}^N}\int_{\mathbb{R}^N}
\frac{|v_{n}(x)|^{2_{\mu}^{\ast}}|v_{n}(y)|^{2_{\mu}^{\ast}}}
{|x-y|^{\mu}}dxdy\right)^{\frac{N-2}{2N-\mu}}\\
&\leq\int_{\mathbb{R}^N}|\nabla v_{n}|^{2}dx\\
&\leq\int_{\mathbb{R}^N}|\nabla v_{n}|^{2}dx+\int_{F^{c}} \lambda_{n} V(x)|v_{n}|^{2}dx-\beta\int_{F^{c}} |v_{n}|^{2}dx\\
&=\int_{\mathbb{R}^{N}}(|\nabla v_{n}|^{2}dx+\lambda_{n} V|v_{n}|^{2}-\beta|v_{n}|^{2})dx+\beta\int_{F}|v_{n}|^{2}dx\\
&=\int_{\mathbb{R}^N}\int_{\mathbb{R}^N}
\frac{|v_{n}(x)|^{2_{\mu}^{\ast}}|v_{n}(y)|^{2_{\mu}^{\ast}}}
{|x-y|^{\mu}}dxdy+o_{n}(1),\\
\endaligned
$$
thanks to \eqref{B1} and \eqref{D3}. It follows that
$$
S_{H,L}\leq\left(\int_{\mathbb{R}^N}\int_{\mathbb{R}^N}
\frac{|v_{n}(x)|^{2_{\mu}^{\ast}}|v_{n}(y)|^{2_{\mu}^{\ast}}}
{|x-y|^{\mu}}dxdy\right)^{\frac{N-\mu+2}{2N-\mu}}
+o_{n}(1)\leq\left(\int_{\mathbb{R}^N}\int_{\mathbb{R}^N}
\frac{|u_{n}(x)|^{2_{\mu}^{\ast}}|u_{n}(y)|^{2_{\mu}^{\ast}}}
{|x-y|^{\mu}}dxdy\right)^{\frac{N-\mu+2}{2N-\mu}}+o_{n}(1)
$$
and so, by \eqref{D1},
$$
S_{H,L}^{\frac{2N-\mu}{N-\mu+2}}\leq\lim\limits_{n\rightarrow\infty}\int_{\mathbb{R}^N}\int_{\mathbb{R}^N}
\frac{|u_{n}(x)|^{2_{\mu}^{\ast}}|u_{n}(y)|^{2_{\mu}^{\ast}}}
{|x-y|^{\mu}}dxdy<S_{H,L}^{\frac{2N-\mu}{N-\mu+2}}.
$$
This is a contradiction and consequently $$\displaystyle\int_{\mathbb{R}^N}\int_{\mathbb{R}^N}
\frac{|v_{n}(x)|^{2_{\mu}^{\ast}}|v_{n}(y)|^{2_{\mu}^{\ast}}}
{|x-y|^{\mu}}dxdy\rightarrow0.$$ From \eqref{D3} we get
\begin{equation}\label{D41}
\langle L_{\lambda_{n},\beta}v_{n},v_{n}\rangle\rightarrow0.
\end{equation}
Hence, by \eqref{D2}
\begin{equation}\label{D4}
\lim\limits_{n\rightarrow\infty}\langle L_{\lambda_{n},\beta}u_{n},u_{n}\rangle=\langle L_{0,\beta}u,u\rangle.
\end{equation}
Recall that $\displaystyle\int_{\mathbb{R}^N}|\nabla v_{n}|^{2}dx\geq\beta\displaystyle\int_{\mathbb{R}^N}| v_{n}|^{2}dx$, we know
$$\aligned
\int_{\mathbb{R}^N}V|u_{n}|^{2}dx&\leq\int_{\mathbb{R}^N}\lambda_{n}V|u_{n}|^{2}dx\\
&=\int_{\mathbb{R}^N}\lambda_{n}V|v_{n}|^{2}dx\\
&\leq\langle L_{\lambda_{n},\beta}v_{n},v_{n}\rangle,\\
\endaligned
$$
since $u_{n}=v_{n}$ in $\mathbb{R}^N\backslash\Omega$ and $V=0$ for $x\in\Omega$.
Combining this with \eqref{D41}, we know $\displaystyle\int_{\mathbb{R}^N}V|u_{n}|^{2}dx\rightarrow0$  and obtain from \eqref{D4} that $u_{n}\rightarrow u$ in $E$.
$\hfill{} \Box$

\begin{Rem} From Theorem \ref{EXS}, we know that for every $0<\beta<\beta_{1}$ there exists $\lambda_{\beta}>0$ such that, for each $\lambda\geq\lambda_{\beta}$, equation \eqref{CCE} has at least one ground state solution $u$. Let $\lambda_{n}\geq\lambda_{\beta}$ and $\lambda_{n}\rightarrow\infty$, we denote $\{u_{n}\}$ be a sequence of ground state solutions of \eqref{CCE} with $\lambda=\lambda_{n}$. By Proposition \ref{MPE4}, we have $J_{\lambda_{n},\beta}(u_{n})\leq c(\beta,\Omega)<\frac{N+2-\mu}{4N-2\mu}S_{H,L}^{\frac{2N-\mu}{N+2-\mu}}$. It is easy to see that $\{u_{n}\}$ is bounded in $E$, $\lambda_{n}\rightarrow\infty$ and $J_{\lambda_{n},\beta}(u_{n})\rightarrow c<\frac{N+2-\mu}{4N-2\mu}S_{H,L}^{\frac{2N-\mu}{N+2-\mu}}$.
\end{Rem}

\begin{Rem}\label{LEX} From the main results in \cite{GY}, we can see that $c(\beta,\Omega)$ can also be achieved by a function, hereafter it will be denoted by $u_{\beta}$ on the Nehari manifold ${\cal M}_{\beta,\Omega}$. Moreover, we add a subscript $r$ to denote the same quantities when the domain $\Omega$
is replaced by $B_{r}\subset \Omega$. Then, $c(\beta,B_{r})$ can also be achieved by some $u_{\beta,B_{r}}\in {\cal M}_{\beta,B_{r}}$.
\end{Rem}

\section{Multiplicity of solutions for the case $0<\beta<\beta_{1}$}
We consider
\begin{equation}\label{CCCE}
\left\{\begin{array}{l}
\displaystyle-\Delta u
=\big(|x|^{-\mu}\ast |u|^{2_{\mu}^{\ast}}\big)|u|^{2_{\mu}^{\ast}-2}u\hspace{4.14mm}\mbox{in}\hspace{1.14mm} \mathbb{R}^N,\\
\displaystyle u\in D^{1,2}(\mathbb{R}^N).
\end{array}
\right.
\end{equation}
The functional associated to \eqref{CCCE} is
$$
J_{\ast}(u)=\frac{1}{2}\int_{\mathbb{R}^N}|\nabla u|^{2}dx-\frac{1}{2\cdot2_{\mu}^{\ast}}\int_{\mathbb{R}^N}\int_{\mathbb{R}^N}
\frac{|u(x)|^{2_{\mu}^{\ast}}|u(y)|^{2_{\mu}^{\ast}}}
{|x-y|^{\mu}}dxdy
$$
and its Nehari manifold is
$$
{\cal M}_{\ast}=\{u\in D^{1,2}(\mathbb{R}^N)\backslash\{{0}\}:\langle J_{\ast}'(u),u\rangle=0\}.
$$
Set
$$
c_{\ast}:=\inf\limits_{u\in {\cal M}_{\ast}}J_{\ast}(u).
$$
We can get
\begin{equation}\label{E2}
c_{\ast}=\inf\limits_{u\in D^{1,2}(\mathbb{R}^N)),u\neq0}\max\limits_{t\geq0}J_{\ast}(tu),\end{equation}
moreover, since $\tilde{U}(x)$ is the unique solution, we know
\begin{equation}\label{E5}
c_{\ast}=\frac{N+2-\mu}{4N-2\mu}S_{H,L}^{\frac{2N-\mu}{N+2-\mu}}.
\end{equation}

\begin{lem}\label{MS1}
Let $0<\beta<\beta_{1}$ and $t_{\beta}$ be the unique value such that $t_{\beta}u_{\beta}\in {\cal M}_{\ast}$. Then
$$
\lim_{\beta\rightarrow0}t_{\beta}=1.
$$
Where $u_{\beta}$ is defined in Remark \ref{LEX}.
\end{lem}
\begin{proof}
By the definition of ${\cal M}_{\ast}$, $t_{\beta}$ satisfies
$$
t_{\beta}^{2}\int_{\mathbb{R}^N}|\nabla u_{\beta}|^{2}dx=t_{\beta}^{2\cdot2_{\mu}^{\ast}}\int_{\mathbb{R}^N}\int_{\mathbb{R}^N}
\frac{|u_{\beta}(x)|^{2_{\mu}^{\ast}}|u_{\beta}(y)|^{2_{\mu}^{\ast}}}
{|x-y|^{\mu}}dxdy.
$$
By Remark \ref{LEX}, we have
$$
\int_{\mathbb{R}^N}|\nabla u_{\beta}|^{2}dx=\beta\int_{\mathbb{R}^N} |u_{\beta}|^{2}dx+\int_{\mathbb{R}^N}\int_{\mathbb{R}^N}
\frac{|u_{\beta}(x)|^{2_{\mu}^{\ast}}|u_{\beta}(y)|^{2_{\mu}^{\ast}}}
{|x-y|^{\mu}}dxdy.
$$
From the two equalities above, we get
$$
\lim_{\beta\rightarrow0}\sup t_{\beta}\geq1.
$$
By Proposition \ref{MPE4}, we have
$$\aligned
(\frac{1}{2}-\frac{1}{2\cdot2_{\mu}^{\ast}})&\left[\int_{\mathbb{R}^N}|\nabla u_{\beta}|^{2}dx-\beta\int_{\mathbb{R}^N} |u_{\beta}|^{2}dx\right]\\
&=J_{\beta,\Omega}(u_{\beta})-\frac{1}{2\cdot2_{\mu}^{\ast}}\langle J_{\beta,\Omega}'(u_{\beta}),u_{\beta}\rangle\\
&\leq \frac{N+2-\mu}{4N-2\mu}S_{H,L}^{\frac{2N-\mu}{N+2-\mu}}.
\endaligned
$$
Since $0<\beta<\beta_{1}$, $\displaystyle\int_{\mathbb{R}^N}|\nabla u_{\beta}|^{2}dx$ is bounded uniformly in $\beta$. Then,
$$
\int_{\mathbb{R}^N}|\nabla u_{\beta}|^{2}dx=\int_{\mathbb{R}^N}\int_{\mathbb{R}^N}
\frac{|u_{\beta}(x)|^{2_{\mu}^{\ast}}|u_{\beta}(y)|^{2_{\mu}^{\ast}}}
{|x-y|^{\mu}}dxdy+o_{\beta}(1)
$$
as $\beta>0$ small enough. Thus, $\lim_{\beta\rightarrow0}t_{\beta}=1.$
\end{proof}

Without loss of generality, we may assume that $B_{\delta}\subset\Omega\subset B_{\kappa_{0}\delta}$ for some positive $\kappa_{0}$. Consider a cut-off function $\psi\in C_{0}^{\infty}(\mathbb{R}^N)$ such that
$$\aligned
\psi(x)=1 \hspace{3.14mm} \mbox{if}\hspace{2.14mm} |x|\leq \delta,\ \ \psi(x)=0 \hspace{3.14mm} \mbox{if} \hspace{2.14mm}|x|\geq2 \delta.
\endaligned
$$
We define, for $\varepsilon>0$,
\begin{equation}\label{E55}
\aligned
U_{\varepsilon}(x)&:=\varepsilon^{\frac{2-N}{2}}U(\frac{x}{\varepsilon}),\\
u_{\varepsilon}(x)&:=\psi(x)U_{\varepsilon}(x),
\endaligned
\end{equation}
where $U$ is defined in introduction. From \cite{GY, GY3} and Lemma 1.46 of \cite{Wi}, we know that
as $\varepsilon\rightarrow0^{+}$,

\begin{equation}\label{E56}
\int_{\mathbb{R}^N}|\nabla u_{\varepsilon}|^{2}dx
=C(N,\mu)^{\frac{N-2}{2N-\mu}\cdot\frac{N}{2}}S_{H,L}^{\frac{N}{2}}+O(\varepsilon^{N-2}),
\end{equation}
\begin{equation}\label{E57}
\aligned
\left(\int_{\R^N}\int_{\R^N}\frac{|u_{\varepsilon}(x)|^{2_{\mu}^{\ast}}|u_{\varepsilon}(y)|^{2_{\mu}^{\ast}}}
{|x-y|^{\mu}}dxdy\right)^{\frac{N-2}{2N-\mu}}
&\leq C(N,\mu)^{\frac{N-2}{2N-\mu}\cdot\frac{N}{2}}S_{H,L}^{\frac{N-2}{2}}+O(\varepsilon^{N-2})
\endaligned
\end{equation}
and
\begin{equation}\label{E58}
\int_{\mathbb{R}^N}|u_{\varepsilon}|^{2}dx=\left\{\begin{array}{l}
\displaystyle d\varepsilon^{2}|\ln\varepsilon|+O(\varepsilon^{2}) \hspace{10.64mm} \mbox{if}\hspace{2.14mm} N=4,\\
\displaystyle d\varepsilon^{2}+O(\varepsilon^{N-2}) \hspace{13.14mm} \mbox{if} \hspace{2.14mm}N\geq5,\\
\end{array}
\right.
\end{equation}
where $d$ is a positive constant.

\begin{lem}\label{MS2} $\lim_{\beta\rightarrow0} c_{\beta,\Omega}=c_{\ast}.$ Where $c_{\beta,\Omega}$ is defined as in \eqref{C3}.
\end{lem}
\begin{proof}
If $N\geq5$, by \eqref{E55} to \eqref{E58}, for $\varepsilon$ small enough, we have

$$\aligned
\max_{t\geq0}J_{\beta,\Omega}(tu_{\varepsilon})&=\max_{t\geq0}\left(\frac{t^{2}}{2}\int_{\Omega}|\nabla u_{\varepsilon}|^{2}dx-\frac{\beta t^{2}}{2}\int_{\Omega} |u_{\varepsilon}|^{2}dx-\frac{t^{2\cdot2_{\mu}^{\ast}}}{2\cdot2_{\mu}^{\ast}}\int_{\mathbb{R}^N}\int_{\mathbb{R}^N}
\frac{|u_{\varepsilon}(x)|^{2_{\mu}^{\ast}}|u_{\varepsilon}(y)|^{2_{\mu}^{\ast}}}
{|x-y|^{\mu}}dxdy\right)\\
&=\frac{N+2-\mu}{4N-2\mu}\left[\frac{\displaystyle\int_{\Omega}|\nabla u_{\varepsilon}|^{2}dx-\beta\int_{\Omega}| u_{\varepsilon}|^{2}dx}{\left(\displaystyle\int_{\mathbb{R}^N}\int_{\mathbb{R}^N}
\frac{|u_{\varepsilon}(x)|^{2_{\mu}^{\ast}}|u_{\varepsilon}(y)|^{2_{\mu}^{\ast}}}
{|x-y|^{\mu}}dxdy\right)^{\frac{N-2}{2N-\mu}}}\right]^{\frac{2N-\mu}{N+2-\mu}}\\
&\geq\frac{N+2-\mu}{4N-2\mu}\left[\frac{C(N,\mu)^{\frac{N-2}{2N-\mu}\cdot\frac{N}{2}}
S_{H,L}^{\frac{N}{2}}+O(\varepsilon^{N-2})-\beta O(\varepsilon^{2})}{C(N,\mu)^{\frac{N-2}{2N-\mu}\cdot\frac{N}{2}}
S_{H,L}^{\frac{N-2}{2}}+O(\varepsilon^{N-2})}\right]^{\frac{2N-\mu}{N+2-\mu}}\\
&=\frac{N+2-\mu}{4N-2\mu}\big[S_{H,L}-\beta O(\varepsilon^{2})\big]^{\frac{2N-\mu}{N+2-\mu}}.\\
\endaligned
$$
Then, we have
$$
\lim_{\beta\rightarrow0}\max_{t\geq0}J_{\beta,\Omega}(tu_{\varepsilon})\geq\frac{N+2-\mu}{4N-2\mu}S_{H,L}^{\frac{2N-\mu}{N+2-\mu}}
$$
for $\varepsilon$ small enough. Similarly, if $N=4$, we have
$$
\lim_{\beta\rightarrow0}\max_{t\geq0}J_{\beta,\Omega}(tu_{\varepsilon})\geq
\frac{6-\mu}{16-2\mu}S_{H,L}^{\frac{8-\mu}{6-\mu}}
$$
for $\varepsilon$ small enough. So, by \eqref{C3}, we get
$$
\lim_{\beta\rightarrow0} c_{\beta,\Omega}\geq\frac{N+2-\mu}{4N-2\mu}S_{H,L}^{\frac{2N-\mu}{N+2-\mu}},
$$
that is,
$$
\lim_{\beta\rightarrow0} c_{\beta,\Omega}\geq c_{\ast}.
$$
On the other hand, by Proposition \ref{MPE4} and \eqref{E5}, we already have
$$
c_{\beta,\Omega}< c_{\ast}
$$
for every $0<\beta<\beta_{1}$. Hence the conclusion follows.
\end{proof}

To prove Theorem \ref{EXS3}, we follow the idea in \cite{BC}. The barycenter of function $u\in H_{0}^{1}(\Omega)$ (see \cite{BC}) is defined as
$$
\alpha(u)=\frac{\displaystyle\int_{\mathbb{R}^N}x|\nabla u|^{2}dx}{\displaystyle\int_{\mathbb{R}^N}|\nabla u|^{2}dx}.
$$
Since $\Omega$ is a bounded smooth domain of $\mathbb{R}^N$, we may fix $r>0$ small enough such that
$$
\Omega_{2r}^{+}=\{x\in \mathbb{R}^N: d(x,\Omega)\leq 2r\}
$$
and
$$
\Omega_{r}^{-}=\{x\in \Omega: d(x,\partial\Omega)\geq r\}
$$
are homotopically equivalent to $\Omega$. In particular we denote by
$$
h:\Omega_{2r}^{+}\rightarrow\Omega_{r}^{-}
$$
the homotopic equivalence map such that $h|_{\Omega_{r}^{-}}$ is the identity.

\begin{lem}\label{MS3}
Let $\{u_{n}\}\subset H_{0}^{1}(\Omega)$ be a $(PS)$ sequence for $J_{\ast}$ at level $c_{\ast}=\frac{N+2-\mu}{4N-2\mu}S_{H,L}^{\frac{2N-\mu}{N+2-\mu}}$. Then, for some subsequence of $\{u_{n}\}$, still denoted
by itself, such that\\
(i) $\{u_{n}\}$ has a subsequence strongly convergent in $D^{1,2}(\mathbb{R}^N)$; or\\
(ii) there exists $\{y_{n}\}\subset\Omega$ such that the sequence $v_{n}(x)=u_{n}(x+y_{n})$ converges strongly in $D^{1,2}(\mathbb{R}^N)$.
\end{lem}
\begin{proof}
By \eqref{B6} with $\lambda=0$ and $\beta=0$, we know the sequence $\{u_{n}\}$ is bounded in $H_{0}^{1}(\Omega)$. Hence, there exists $u\in H_{0}^{1}(\Omega)$ such that $u_{n}\rightharpoonup u$ in $H_{0}^{1}(\Omega)$, up to some subsequence. We next continue our arguments by distinguishing two
cases: $u\neq0$ and $u=0$.

\textbf{Case 1.} $u\neq0$.

In this case, since $\{u_{n}\}\subset H_{0}^{1}(\Omega)$ is bounded and $u_{n}\rightharpoonup u$ in $H_{0}^{1}(\Omega)$, we have
\begin{equation}\label{E6}
\int_{\mathbb{R}^N}|\nabla u|^{2}dx\leq\lim_{n\rightarrow\infty}\int_{\mathbb{R}^N}|\nabla u_{n}|^{2}dx.
\end{equation}
Since $\{u_{n}\}$ is a $(PS)$ sequence for $J_{\ast}$, we can get that $\langle J_{\ast}'(u),u \rangle=0$. Observe that we must have the equality in \eqref{E6}. Otherwise, by Fatou's lemma,
$$\aligned
c_{\ast}&\leq J_{\ast}(u)\\
&=J_{\ast}(u)-\frac{1}{2\cdot2_{\mu}^{\ast}}\langle J_{\ast}'(u),u \rangle\\
&=\frac{N+2-\mu}{4N-2\mu}\int_{\mathbb{R}^N}|\nabla u|^{2}dx\\
&\leq\lim_{n\rightarrow\infty}\frac{N+2-\mu}{4N-2\mu}\int_{\mathbb{R}^N}|\nabla u_{n}|^{2}dx\\
&=\lim_{n\rightarrow\infty}(J_{\ast}(u_{n})-\frac{1}{2\cdot2_{\mu}^{\ast}}\langle J_{\ast}'(u_{n}),u_{n} \rangle)\leq c_{\ast},\\
\endaligned
$$
which leads to a contradiction. Thus, up to subsequences, we have
$$
\lim_{n\rightarrow\infty}\int_{\mathbb{R}^N}|\nabla u_{n}|^{2}dx\rightarrow\int_{\mathbb{R}^N}|\nabla u|^{2}dx.
$$
Hence, $\{u_{n}\}$ has a subsequence which convergent to $u$ strongly in $D^{1,2}(\mathbb{R}^N)$.

\textbf{Case 2.} $u=0$.

Since $\{u_{n}\}$ is a $(PS)$ sequence for $J_{\ast}$, we get
$$\aligned
J_{\ast}(u_{n})&=J_{\ast}(u_{n})-\frac{1}{2\cdot2_{\mu}^{\ast}}\langle J_{\ast}'(u_{n}),u_{n} \rangle+o_{n}(1)\\
&=\frac{N+2-\mu}{4N-2\mu}\int_{\mathbb{R}^N}|\nabla u_{n}|^{2}dx+o_{n}(1)\rightarrow\frac{N+2-\mu}{4N-2\mu}S_{H,L}^{\frac{2N-\mu}{N+2-\mu}}.\\
\endaligned
$$
Then $\displaystyle\int_{\mathbb{R}^N}|\nabla u_{n}|^{2}dx\not\rightarrow0$. So, there exist $r,\delta>0$ such that
$$
\lim_{n\rightarrow\infty}\sup_{y\in \mathbb{R}^N}\int_{B_{r}(y)}|\nabla u_{n}|^{2}dx\geq\delta.
$$
Otherwise, we have $\nabla u_{n}\rightarrow0$ in $L^{p}( \mathbb{R}^N)$ with $2<p<2^{\ast}$ from the concentration compactness
principle (see Lemma 1.21 of \cite{Wi}). Since $\{u_{n}\}\subset H_{0}^{1}(\Omega)$ and $\Omega$ is bounded, we can deduce $\nabla u_{n}\rightarrow0$ in $L^{2}(\mathbb{R}^N)$, which contradicts to the fact that $\displaystyle\int_{\mathbb{R}^N}|\nabla u_{n}|^{2}dx\not\rightarrow0$. So, there exist $r,\delta>0$ and $\{y_{n}\}\subset\mathbb{R}^N$ such that
$$
\lim_{n\rightarrow\infty}\sup\int_{B_{r}(y_{n})}|\nabla u_{n}|^{2}dx\geq\delta.
$$
Since $\mbox{supp} u_{n}\subset\Omega$, we can choose $\{y_{n}\}\subset\Omega$. Let $v_{n}(x)=u_{n}(x+y_{n})$, then $J_{\ast}(v_{n})\rightarrow\frac{N+2-\mu}{4N-2\mu}S_{H,L}^{\frac{2N-\mu}{N+2-\mu}}$ and $\langle J_{\ast}'(v_{n}),v_{n} \rangle\rightarrow0$. It is clear that $v_{n}$ is bounded in $D^{1,2}(\mathbb{R}^N)$ and there exists $v\in D^{1,2}(\mathbb{R}^N)$ with $v\neq0$ such that
$v_{n}\rightharpoonup v$ in $D^{1,2}(\mathbb{R}^N)$. Then, the proof follows from the arguments used in Case 1.
\end{proof}

From Proposition \ref{ExFu}, we know that functions of type
$$
U_{b}(\cdot-a)=\frac{C_{0}(b^{2})^{\frac{N-2}{4}}}{(b^{2}+|x-a|^{2})^{\frac{N-2}{2}}}, \hspace{3.14mm}\mbox{for some} \ \ C_{0},b\in\mathbb{R} \hspace{3.14mm}\mbox{and} \hspace{3.14mm}a\in \mathbb{R}^{N}
$$
achieves the minimum of $J_{\ast}$ on ${\cal M}_{\ast}$ and the minimum value is exactly
$$
J_{\ast}(U_{b}(\cdot-a))=\frac{N+2-\mu}{4N-2\mu}\int_{\mathbb{R}^{N}}|\nabla U_{b}(\cdot-a)|^{2}dx=\frac{N+2-\mu}{4N-2\mu}S_{H,L}^{\frac{2N-\mu}{N+2-\mu}}=c_{\ast}.
$$
Let $\{u_{n}\}\subset H_{0}^{1}(\Omega)$ be a $(PS)$ sequence for $J_{\ast}$ at level $c_{\ast}$. Then, Lemma \ref{MS3} implies that
$$
\lim_{n\rightarrow\infty}\int_{\mathbb{R}^N}\big|\nabla \big(u_{n}-\frac{C_{0}(b_{1,n}^{2})^{\frac{N-2}{4}}}{(b_{1,n}^{2}+|x-x_{1,n}|^{2})^{\frac{N-2}{2}}}\big)\big|^{2}dx\rightarrow0
$$
or
$$
\lim_{n\rightarrow\infty}\int_{\mathbb{R}^N}\big|\nabla \big(u_{n}(x+y_{n})-\frac{C_{0}(b_{2,n}^{2})^{\frac{N-2}{4}}}{(b_{2,n}^{2}+|x-x_{2,n}|^{2})^{\frac{N-2}{2}}}\big)\big|^{2}dx\rightarrow0,
$$
which means,
$$
\lim_{n\rightarrow\infty}\int_{\mathbb{R}^N}\big|\nabla \big(u_{n}(x)-\frac{C_{0}(b_{2,n}^{2})^{\frac{N-2}{4}}}{(b_{2,n}^{2}+|x-y_{n}-x_{2,n}|^{2})^{\frac{N-2}{2}}}\big)\big|^{2}dx\rightarrow0
$$
for some sequence $b_{1,n}, b_{2,n}\in\mathbb{R}\backslash\{0\}$ and $x_{1,n}, x_{2,n}\in\mathbb{R}^N$. Notice that $\mbox{supp} u_{n}\in\Omega$, we have
$x_{1,n},y_{n}+x_{2,n}\in\Omega$, then there exists some sequence $b_{n}\in\mathbb{R}\backslash\{0\}\rightarrow0$ and $x_{n}\in\Omega$ such that
\begin{equation}\label{E9}
\lim_{n\rightarrow\infty}\int_{\mathbb{R}^N}|\nabla (u_{n}(x)- U_{b_{n}}(x-x_{n}))|^{2}dx\rightarrow0.
\end{equation}
What's more, we can observe from
\begin{equation}\label{E7}
\lim_{n\rightarrow\infty}\int_{\mathbb{R}^N\backslash \Omega}\big|\nabla \frac{C_{0}(b_{n}^{2})^{\frac{N-2}{4}}}{(b_{n}^{2}+|x-x_{n}|^{2})^{\frac{N-2}{2}}}\big|^{2}dx\rightarrow0
\end{equation}
that $b_{n}\rightarrow0$ as $n$ goes to infinity.

\begin{Prop}\label{MS4} There exists $\beta^{*}=\beta^{*}(r)\in(0, \beta_{1})$ such that, for $0<\beta\leq\beta^{*}$, $\alpha(u)\in \Omega_{r}^{+}$ for every $u\in {\cal M}_{\beta,\Omega}$ with $J_{\beta,\Omega}(u)\leq c(\beta,B_{r})$.
\end{Prop}
\begin{proof}
As in \cite{Si}, we argue by contradiction. Assume that there exist sequences $\varepsilon_{n}\rightarrow0$, $\beta_{n}\rightarrow0$ and $u_{n}\in {\cal M}_{\beta_{n},\Omega}$ such that
$$
J_{\beta_{n},\Omega}(u_{n})< c(\beta_{n},B_{r})+\varepsilon_{n}\hspace{3.14mm} \mbox{and}\hspace{3.14mm} \alpha(u_n)\not\in \Omega_{r}^{+}.
$$
Then, by Lemma \ref{MS2}, we have $J_{\beta_{n},\Omega}(u_{n})\rightarrow c_{\ast}$ and $\{u_{n}\}$ is bounded in $H_{0}^{1}(\Omega)$. Let $t_{n}$ such that $t_{n}u_{n}\in {\cal M}_{\ast}$. Using Lemma \ref{MS1} and $u_{n}\in {\cal M}_{\beta_{n},\Omega}$, we know $t_{n}\to 1$. Thanks to $J_{\beta_{n},\Omega}(u_{n})\rightarrow c_{\ast}$, we know
$$\aligned
J_{\beta,\Omega}(u_{n})&-J_{\ast}(t_{n}u_{n})\\
&=\frac{N+2-\mu}{4N-2\mu}(1-t_{n}^{2})\int_{\mathbb{R}^N}|\nabla u_{n}|^{2}dx-(\frac{\beta}{2}-\frac{\beta}{2\cdot2_{\mu}^{\ast}})\int_{\mathbb{R}^N}|u_{n}|^{2}dx=o_{n}(1),\\
\endaligned
$$
leads to the fact that $J_{\ast}(t_{n}u_{n})\rightarrow c_{\ast}$. Thus, $\{t_{n}u_{n}\}$ is a $(PS)$ sequence for $J_{\ast}$ at level $c_{\ast}$. By \eqref{E9}, we have
$$
t_{n}u_{n}-U_{b_{n}}(x-x_{n})\rightarrow0 \hspace{3.14mm} \mbox{in}\hspace{3.14mm}D^{1,2}(\mathbb{R}^N)
$$
for some sequence $b_{n}\in\mathbb{R}\backslash\{0\}$ and $x_{n}\in\Omega$. Then, we can write
$$
t_{n}u_{n}=U_{b_{n}}(\cdot-x_{n})-v_{n},
$$
where $v_{n}$ such that $\displaystyle\int_{\mathbb{R}^N}|\nabla v_{n}|^{2}dx\rightarrow0$ and $U_{b_{n}}(\cdot-x_{n})=v_{n}$ on $\mathbb{R}^N\backslash \Omega$. We write $x\in \mathbb{R}^N$ as $x=(x_{(1)},x_{(2)},\cdot\cdot\cdot,x_{(N)})$, the $i$-th coordinate of the barycenter of $u_{n}$ satisfies
\begin{equation}\label{E10}
\aligned
&\alpha(u_{n})_{(i)}\int_{\mathbb{R}^N}|\nabla (t_{n}u_{n})|^{2}dx\\
&\hspace{5mm}=\int_{\mathbb{R}^N}x_{(i)}|\nabla U_{b_{n}}(\cdot-x_{n})|^{2}dx+\int_{\mathbb{R}^N}x_{(i)}|\nabla v_{n}|^{2}dx-2\int_{\mathbb{R}^N}x_{(i)}\nabla U_{b_{n}}(\cdot-x_{n})\nabla v_{n}dx.
\endaligned\end{equation}
Using $U_{b_{n}}(\cdot-x_{n})=v_{n}$ on $\mathbb{R}^N\backslash \Omega$, we have
\begin{equation}\label{E11}
\aligned
&\alpha(u_{n})_{(i)}\int_{\mathbb{R}^N}|\nabla (t_{n}u_{n})|^{2}dx\\
&\hspace{5mm}=\int_{\Omega}x_{(i)}|\nabla U_{b_{n}}(\cdot-x_{n})|^{2}dx+\int_{\Omega}x_{(i)}|\nabla v_{n}|^{2}dx-2\int_{\Omega}x_{(i)}\nabla U_{b_{n}}(\cdot-x_{n})\nabla v_{n}dx\\
&\hspace{5mm}=A_{n}+B_{n}-2D_{n}.
\endaligned
\end{equation}
By simple computations, we know that
\begin{equation}\label{E12}
A_{n}=b_{n}\int_{\Omega_{n}'}y_{(i)}|\nabla U_{1}(y)|^{2}dy+(x_{n})_{(i)}\int_{\Omega_{n}'}|\nabla U_{1}(y)|^{2}dy,
\end{equation}
where $\Omega_{n}'=\{y\in\mathbb{R}^N:y=x-x_{n},x\in\Omega\}$.
Since $b_{n}\rightarrow0$, we get $b_{n}\displaystyle\int_{\Omega_{n}'}y_{i}|\nabla U_{1}(y)|^{2}dy=o_{n}(1)$. From $\displaystyle\int_{\mathbb{R}^N}|\nabla v_{n}|^{2}dx\rightarrow0$, we get $B_{n}=o_{n}(1)$. Since
$$
\int_{\Omega}x_{(i)}\nabla U_{b_{n}}(\cdot-x_{n})\nabla v_{n}dx\leq C \Big(\int_{\Omega}|\nabla U_{b_{n}}(\cdot-x_{n})|^{2}dx\Big)^{\frac{1}{2}}\Big(\int_{\Omega}|\nabla v_{n}|^{2}dx\Big)^{\frac{1}{2}},
$$
then, $D_{n}=o_{n}(1)$. We know that $\displaystyle\int_{\mathbb{R}^N}|\nabla (t_{n}u_{n})|^{2}dx=\displaystyle\int_{\mathbb{R}^N}|\nabla U_{1}(x)|^{2}dx+o_{n}(1)$. In fact, we have shown that
\begin{equation}\label{E13}
\alpha(u_{n})_{(i)}=\frac{(x_{n})_{(i)}\displaystyle\int_{\Omega_{n}'}|\nabla U_{1}(x)|^{2}dx+o_{n}(1)}{\displaystyle\int_{\mathbb{R}^N}|\nabla U_{1}(x)|^{2}dx+o_{n}(1)}.
\end{equation}
Since $x_{n}\in\Omega$ and $\Omega_{n}'\subset\mathbb{R}^N$, \eqref{E13} implies that $\alpha(u_{n})\in\overline{\Omega}$ which is in contrast with assumption and proves the proposition.
\end{proof}

We choose $R>0$ such that $\overline{\Omega}\subset B_{R}$ and set
$$
\eta(t)=\left\{\begin{array}{l}
\displaystyle 1 \hspace{17.14mm} \mbox{if}\hspace{2.14mm} 0\leq t\leq R,\\
\displaystyle R/t \hspace{13.14mm} \mbox{if} \hspace{2.14mm}R\leq t.\\
\end{array}
\right.
$$
On $D^{1,2}(\mathbb{R}^N)$ we define
\begin{equation}\label{BCu}
\alpha_{c}(u)=\frac{\displaystyle\int_{\mathbb{R}^N}x\eta(|x|)|\nabla u|^{2}dx}{\displaystyle\int_{\mathbb{R}^N}|\nabla u|^{2}dx}.
\end{equation}

\begin{Prop}\label{MS5} There exist $0<\beta^{*}<\beta_{1}$ and for each $0<\beta\leq\beta^{*}$ a number $\lambda_{\beta}\geq \lambda_{\beta}$ such that, $\alpha_{c}(u)\in \Omega_{2r}^{+}$ for every $\lambda\geq\lambda_{\beta}$ and $u\in {\cal M}_{\lambda,\beta}$ with $J_{\lambda,\beta}(u)\leq c(\beta,B_{r})$.
\end{Prop}
\begin{proof}
Due to the appearance of the convolution part, we adapt the arguments in \cite{CD1} to suit the new situation. Assume by contradiction that, for $\beta>0$ arbitrarily small, there is a
sequence $\{u_{n}\}\subset {\cal M}_{\lambda_{n},\beta} $ such that $\lambda_{n}\rightarrow\infty$, $J_{\lambda_{n},\beta}(u_{n})\rightarrow c\leq c(\beta,B_{r})$ and $\alpha_{c}(u_{n})\not\in \Omega_{2r}^{+}$. By the proof of Proposition \ref{PS}, we know $\{u_{n}\}$ is bounded in $E$. By Lemma \ref{EMB}, there is a $v_{\beta}\in H_{0}^{1}(\Omega)$ such that, up to a subsequence, $u_{n}\rightharpoonup v_{\beta}$ in $E$ and $u_{n}\rightarrow v_{\beta}$ in $L^{2}(\mathbb{R}^N)$.  Next we continue the proof by distinguishing two
cases: $$\displaystyle\int_{\mathbb{R}^N}\int_{\mathbb{R}^N}
\frac{|v_{\beta}(x)|^{2_{\mu}^{\ast}}|v_{\beta}(y)|^{2_{\mu}^{\ast}}}
{|x-y|^{\mu}}dxdy\leq\langle L_{0,\beta}v_{\beta},v_{\beta}\rangle$$ and $$\displaystyle\int_{\mathbb{R}^N}\int_{\mathbb{R}^N}
\frac{|v_{\beta}(x)|^{2_{\mu}^{\ast}}|v_{\beta}(y)|^{2_{\mu}^{\ast}}}
{|x-y|^{\mu}}dxdy>\langle L_{0,\beta} v_{\beta},v_{\beta}\rangle.$$

\textbf{Case 1. } $\displaystyle\int_{\mathbb{R}^N}\int_{\mathbb{R}^N}
\frac{|v_{\beta}(x)|^{2_{\mu}^{\ast}}|v_{\beta}(y)|^{2_{\mu}^{\ast}}}
{|x-y|^{\mu}}dxdy\leq\langle L_{0,\beta} v_{\beta},v_{\beta}\rangle$.

Since $\{u_{n}\}\subset {\cal M}_{\lambda_{n},\beta}$, $\lambda_{n}\rightarrow\infty$ and $J_{\lambda_{n},\beta}(u_{n})\rightarrow c\leq c(\beta,B_{r})<\frac{N+2-\mu}{4N-2\mu}S_{H,L}^{\frac{2N-\mu}{N+2-\mu}}$. We write $v_{n}:=u_{n}-v_{\beta}$. By the proof of the asymptotic behavior of the solutions of \eqref{CCE} in Theorem \ref{EXS}, we know,

$$
\lim_{n\rightarrow\infty}\int_{\mathbb{R}^N}|\nabla u_{n}|^{2}dx\rightarrow\int_{\mathbb{R}^N}|\nabla v_{\beta}|^{2}dx.
$$
Consequently, $\alpha_{c}(u_{n})\rightarrow\alpha(v_{\beta})$. However, $J_{\beta,\Omega}(v_{\beta})\leq \lim_{n\rightarrow\infty}J_{\lambda_{n},\beta}(u_{n})\leq c(\beta,B_{r})$, it follows from Proposition \ref{MS4} that $\alpha(v_{\beta})\in \Omega_{r}^{+}$, this is a contradiction.

\textbf{Case 2. } $\displaystyle\int_{\mathbb{R}^N}\int_{\mathbb{R}^N}
\frac{|v_{\beta}(x)|^{2_{\mu}^{\ast}}|v_{\beta}(y)|^{2_{\mu}^{\ast}}}
{|x-y|^{\mu}}dxdy>\langle L_{0,\beta} v_{\beta},v_{\beta}\rangle$.

By the arguments of Proposition \ref{PS} and $c(\beta,B_{r})<\frac{N+2-\mu}{4N-2\mu}S_{H,L}^{\frac{2N-\mu}{N+2-\mu}}$, we know $\|u_{n}\|_{L_{\lambda_{n},\beta}}$ is bounded uniformly in $0<\beta<\beta_{1}$ and $\lambda\geq\lambda_{\beta}$. Thanks to the proof of Lemma \ref{MS1}, we know $|v_{\beta}|_{2}^{2}$ is bounded uniformly in $0<\beta<\beta_{1}$. Then, $\beta|v_{\beta}|_{2}^{2}=o_{\beta}(1)$ and $\beta|u_{n}|_{2}^{2}=o_{\beta}(1)$ for $\beta>0$ small enough. It is easy to see that there exists $t_{\beta}\in(0,1)$ such that $t_{\beta}v_{\beta}\in {\cal M}_{\beta,\Omega}$. Then, we have
$$
t_{\beta}^{2}\int_{\mathbb{R}^N}|\nabla v_{\beta}|^{2}dx=t_{\beta}^{22_{\mu}^{\ast}}\int_{\mathbb{R}^N}\int_{\mathbb{R}^N}
\frac{|v_{\beta}(x)|^{2_{\mu}^{\ast}}|v_{\beta}(y)|^{2_{\mu}^{\ast}}}
{|x-y|^{\mu}}dxdy+\beta t_{\beta}^{2}\int_{\mathbb{R}^N}|v_{\beta}|^{2}dx.
$$
Combining this with the fact that $\{u_{n}\}\subset {\cal M}_{\lambda_{n},\beta} $ we get
$$
J_{\beta,\Omega}(t_{\beta}v_{\beta})=\frac{N+2-\mu}{4N-2\mu}t_{\beta}^{2}\int_{\mathbb{R}^N}|\nabla v_{\beta}|^{2}dx-\beta\frac{N+2-\mu}{4N-2\mu}t_{\beta}^{2}\int_{\mathbb{R}^N}|v_{\beta}|^{2}dx
$$
and
$$
J_{\lambda_{n},\beta}(u_{n})=\frac{N+2-\mu}{4N-2\mu}\int_{\mathbb{R}^N}(|\nabla u_{n}|^{2}+\lambda_{n}V|u_{n}|^{2})dx-\beta\frac{N+2-\mu}{4N-2\mu}\int_{\mathbb{R}^N}|u_{n}|^{2}dx.
$$
Thus,
$$\aligned
c(\beta,\Omega)+\beta\frac{N+2-\mu}{4N-2\mu}t_{\beta}^{2}\int_{\mathbb{R}^N}|v_{\beta}|^{2}dx
&\leq \frac{N+2-\mu}{4N-2\mu}t_{\beta}^{2}\int_{\mathbb{R}^N}|\nabla v_{\beta}|^{2}dx\\
&\leq \lim_{n\rightarrow\infty}\frac{N+2-\mu}{4N-2\mu}\int_{\mathbb{R}^N}|\nabla u_{n}|^{2}dx\\
&\leq \lim_{n\rightarrow\infty}\frac{N+2-\mu}{4N-2\mu}\int_{\mathbb{R}^N}(|\nabla u_{n}|^{2}+\lambda_{n}V|u_{n}|^{2})dx\\
&\leq c(\beta,B_{r})+\beta\frac{N+2-\mu}{4N-2\mu}\int_{\mathbb{R}^N}|u_{n}|^{2}dx.
\endaligned
$$
It follows that, for $n$ large enough,
$$
\left|\int_{\mathbb{R}^N}|\nabla u_{n}|^{2}dx-t_{\beta}^{2}\int_{\mathbb{R}^N}|\nabla v_{\beta}|^{2}dx\right|\leq\frac{4N-2\mu}{N+2-\mu}(c(\beta,B_{r})-c(\beta,\Omega))+o_{\beta}(1).
$$
Since $|c(\beta,B_{r})-c(\beta,\Omega)|\rightarrow0$ as $\beta\rightarrow0$, this implies that $\left|\displaystyle\int_{\mathbb{R}^N}|\nabla u_{n}|^{2}dx-t_{\beta}^{2}\displaystyle\int_{\mathbb{R}^N}|\nabla v_{\beta}|^{2}dx\right|<r$ for all $\beta$ sufficiently small. But, by Proposition \ref{MS4}, there holds $\alpha(t_{\beta}v_{\beta})\in \Omega_{r}^{+}$ which contradicts to the assumption $\alpha_{c}(u_{n})\not\in \Omega_{2r}^{+}$ again.
\end{proof}

For convenience, we denote $J_{\lambda,\beta}^{\leq c(\beta,B_{r})}=\{z\in {\cal M}_{\lambda,\beta}:J_{\lambda,\beta}(z)\leq c(\beta,B_{r})\}$.

\noindent
{\bf Proof of Theorem \ref{EXS3}.}
For $0<\beta\leq\beta^{*}$ and $\lambda\geq\lambda_{\beta}$, we define two maps
$$
\Omega_{r}^{-}\stackrel{\psi_{\beta,r}}{\longrightarrow} J_{\lambda,\beta}^{\leq c(\beta,B_{r})}\stackrel{h\circ\alpha_{c}} {\longrightarrow}\Omega_{r}^{-}
$$
as follow: The map $\alpha_{c}$ is defined in \eqref{BCu} and $
h:\Omega_{2r}^{+}\rightarrow\Omega_{r}^{-}
$
 is the homotopic equivalence map such that $h|_{\Omega_{r}^{-}}$ is the identity. Let $u_{\beta,B_{r}}\in H_{0}^{1}(B_{r})$ be a minimizer of $c(\beta,B_{r})$ on ${\cal M}_{\beta,B_{r}}$. We define the map $\psi_{\beta,r}:\Omega_{r}^{-}\rightarrow {\cal M}_{\lambda,\beta}$ by
$$
\psi_{\beta,r}(y)(x)=\left\{\begin{array}{l}
\displaystyle u_{\beta,B_{r}}(x-y) \hspace{10.14mm} \mbox{if}\hspace{2.14mm} x\in B_{r}(y),\\
\displaystyle 0 \hspace{27.64mm} \mbox{if} \hspace{2.14mm}x\in\Omega\backslash B_{r}(y).\\
\end{array}
\right.
$$
Then, we can see that $\psi_{\beta,r}(y)(x)\equiv0$ in $\mathbb{R}^N\backslash \Omega$ for every $y\in \Omega_{r}^{-}$, it
follows that $\alpha_{c}(\psi_{\beta,r}(y)(x))\in B_{r}(y)$, $\psi_{\beta,r}(y)(x)\subset {\cal M}_{\lambda,\beta}$ and $J_{\lambda,\beta}(\psi_{\beta,r}(y)(x))=J_{\beta,B_{r}}(\psi_{\beta,r}(y)(x))= c(\beta,B_{r})$. Thus $\psi_{\beta,r}$ is also well defined. Moreover, $\alpha_{c}\circ\psi_{\beta,r}$ is the inclusion $\Omega_{r}^{-}\rightarrow\Omega_{2r}^{+}$. Then we know the composite map $h\circ\alpha_{c}\circ\psi_{\beta,r}$ is homotopic to the identity of $\Omega_{r}^{-}$. By a property of the
category, we get
$$
\mbox{cat}_{J_{\lambda,\beta}^{\leq c(\beta,B_{r})}}(J_{\lambda,\beta}^{\leq c(\beta,B_{r})})\geq\mbox{cat}_{\Omega_{r}^{-}}(\Omega_{r}^{-})
$$
(see e.g. \cite{Ja}) and the choice of $r$ gives $\mbox{cat}_{\Omega_{r}^{-}}(\Omega_{r}^{-})=\mbox{cat}_{\overline{\Omega}}(\overline{\Omega})$. It follows from Proposition \ref{PS} that the $(PS)$ condition is verified on ${\cal M}_{\lambda,\beta}$, by applying the Lusternik-Schnirelmann theory (see e.g. \cite{Pa, Wi}) we obtain the existence of at least $\mbox{cat}_{\overline{\Omega}}(\overline{\Omega})$ critical points for $J_{\lambda,\beta}$ on the manifold ${\cal M}_{\lambda,\beta}$ which are the solutions of \eqref{CCE}. The proof is completed. $\hfill{} \Box$

\section{Existence of solutions for the case $\beta>\beta_{1}$ }
In the following we consider the critical Choquard equation \eqref{CCE} with indefinite potential. Assume that, $0<\mu<4$, $N\geq4$, $\beta>\beta_{1}$, $\beta\neq\beta_{j}$ for any $j>1$ and the potential $V(x)$ satisfies $(V_1)$ and $(V_3)$.

As above sections, we still denote the operator $L_{\lambda,\beta}:=-\Delta+\lambda V(x)-\beta$, particularly, $L_{0,\beta}=-\Delta-\beta$. 
In the following we denote by $|L_{\lambda,\beta}|$ the absolute value of operator $L_{\lambda,\beta}$ and let
$
E_{\lambda}=D(|L_{\lambda,\beta}|^{\frac{1}{2}})
$ be the Hilbert space equipped with the inner product $$
( u_{1},u_{2})=\langle|L_{\lambda,\beta}|^{\frac{1}{2}}u_{1},
|L_{\lambda,\beta}|^{\frac{1}{2}}u_{2}\rangle
$$
and the norm
$$
\|u\|_{L_{\lambda,\beta}}=( u,u)^{\frac{1}{2}}.
$$
By conditions $(V_1)$ and $(V_3)$, $E_{\lambda}$
is continuously embedded in $H^{1}(\mathbb{R}^N)$ for $\lambda$ large enough.

By condition $(V_1)$ and Remark \ref{E0}, we know that the zero set of $V(x)$ is a bounded domain in $\mathbb{R}^{N}$ and so we have that inf $\sigma_{e}(L_{\lambda,\beta})\geq\lambda M_{0}$ and $L_{\lambda,\beta}$ has finite
Morse index on $E_{\lambda}$, where $\sigma_{e}(L_{\lambda,\beta})$ denote the essential spectrum of operator $L_{\lambda,\beta}$ in $E_{\lambda}$ and $M_{0}$ is the same constant appeared in Remark \ref{E0}. Thus $E_{\lambda}$ splits as an orthogonal sum $E_{\lambda}=E_{\lambda}^{-}\oplus E_{\lambda}^{0}\oplus E_{\lambda}^{+}$  according to the
negative, zero and positive eigenspace of $L_{\lambda,\beta}$ and dim $E_{\lambda}^{-}\cup E_{\lambda}^{0}<\infty$.
On the other hand, since inf $\sigma_{e}(L_{\lambda,\beta})\geq\lambda M_{0}$, we may assume that
$\zeta_{1}^{\lambda}<\zeta_{2}^{\lambda}<...< \zeta_{k_{\lambda}}^{\lambda}<e(L_{\lambda,\beta})$ be the distinct eigenvalues of $L_{\lambda,\beta}$ in $E_{\lambda}$ and $k_{\lambda}\in\N$ goes to $\infty$ as $\lambda\rightarrow\infty$. The operator $L_{0,\beta}$ has discrete spectrum in $H_{0}^{1}(\Omega)$ and we denote them as
$$\zeta_{1}<\zeta_{2}<\cdot\cdot\cdot<\zeta_{j}<\zeta_{j+1}<\cdot\cdot\cdot,\ \ \zeta_{j}=\beta_{j}-\beta$$
which are the distinct eigenvalues of $L_{0,\beta}$ in
$H_{0}^{1}(\Omega)$. 
 Let $\F_{j}^{\lambda}(j\leq k_{\lambda})$ be the corresponding eigenspaces of $\zeta_{j}^{\lambda}$ and $\F_{j}$ be the corresponding eigenspaces of $\zeta_{j}$. Involving the relationship the eigenspaces, the following two Lemmas are taken from \cite{Tz}.
\begin{lem}  \label{IP1}
$\zeta_{j}^{\lambda}\rightarrow\zeta_{j}$ and $\F_{j}^{\lambda}\rightarrow\F_{j}$ as $\lambda\rightarrow\infty$.
\end{lem}

Here $\F_{j}^{\lambda}\rightarrow\F_{j}$ means that, given any sequence $\lambda_{i}\rightarrow\infty$ and normalized eigenfunctions $\varphi_{i}\in\F_{j}^{\lambda_{i}}$, there exists a normalized eigenfunction $\varphi\in\F_{j}$ such that $\varphi_{i}\rightarrow\varphi$ strongly in $H^{1}(\mathbb{R}^N)$ along a subsequence.

\begin{lem}\label{IP2}
For $\lambda$ large the operator $L_{\lambda,\beta}$ on $E_{\lambda}$ is non-degenerate and has finite Morse index uniformly in $\lambda$.
\end{lem}

By Lemma \ref{IP2}, we can see that for $\lambda$ large, $E_{\lambda}^{0}$ is indeed the zero space $\{0\}$, which implies that for $\lambda$ large, we have $E_{\lambda}=E_{\lambda}^{-}\oplus E_{\lambda}^{+}$.
So, we have
$$
\int_{\mathbb{R}^N}(|\nabla u|^{2}+(\lambda V(x)-\beta)u^{2})dx=\|u^{+}\|_{L_{\lambda,\beta}}^{2}-\|u^{-}\|_{L_{\lambda,\beta}}^{2}
$$
and
$$
J_{\lambda,\beta}(u)=\frac{1}{2}\|u^{+}\|_{L_{\lambda,\beta}}^{2}-\frac{1}{2}\|u^{-}\|_{L_{\lambda,\beta}}^{2}-
\frac{1}{2\cdot2_{\mu}^{\ast}}\int_{\mathbb{R}^N}\int_{\mathbb{R}^N}
\frac{|u(x)|^{2_{\mu}^{\ast}}|u(y)|^{2_{\mu}^{\ast}}}
{|x-y|^{\mu}}dxdy,
$$
where $u=u^{+}+u^{-}\in E_{\lambda}^{+}\oplus E_{\lambda}^{-}$. We define the corresponding Nehari manifold as follows:
$$
{\cal N}_{\lambda}:=\{u\in E_{\lambda}\backslash\{0\}:\langle J_{\lambda,\beta}'(u),u\rangle=0\}.
$$
and denote
\begin{equation}\label{f1}
c_{\lambda}:=\inf_{u\in{\cal N}_{\lambda}}J_{\lambda,\beta}(u).
\end{equation}
In next section, we will show that for $\lambda$ large, \eqref{CCE} admits a ground state solutions $u_{\lambda}$ which achieves $c_{\lambda}$ for $\lambda>0$ large such that $u_{\lambda}$ converge as $\lambda\rightarrow\infty$ towards a ground state solution of \eqref{LCCE} that lies on the level
\begin{equation}\label{f2}
c(\beta,\Omega):=\inf_{u\in{\cal N}_{\beta,\Omega}}J_{\beta,\Omega}(u).
\end{equation}
where ${\cal N}_{\beta,\Omega}:=\{u\in H_{0}^{1}(\Omega)\backslash\{0\}:\langle J_{\beta,\Omega}'(u),u\rangle=0\}$ and
$
J_{\beta,\Omega}
$
is the corresponding variational functional of \eqref{LCCE}, see Section 3.

For $r>0$, we set $B_{r}^{+}=\{u\in E_{\lambda}^{+}:\|u\|_{L_{\lambda,\beta}}\leq r\}$ and $S_{r}^{+}=\{u\in E_{\lambda}^{+}:\|u\|_{L_{\lambda,\beta}}=r\}$, and for $w\in E_{\lambda}^{+}$, we define the convex subset
$$
H_{w}:=\{v+tw:v\in E_{\lambda}^{-},t\geq0\}\subset E_{\lambda}.
$$

\begin{lem}\label{IP33} The functional $J_{\lambda,\beta}$ satisfies the following conditions:\\
(i) There exist $r,\alpha>0$ such that $J_{\lambda,\beta}|_{S_{r}^{+}}(u)\geq\alpha$ and $J_{\lambda,\beta}|_{B_{r}^{+}}(u)\geq0$. \\
(ii) For any $w\in E_{\lambda}^{+}\backslash\{0\}$, there exists $R_{w}>0$ and $C_{w}>0$ such that $J_{\lambda,\beta}(u)<0$ for all $u\in H_{w}\backslash B_{R_{w}}$ and $\max_{u\in H_{w}} J_{\lambda,\beta}(u)\leq C_{w}$.
\end{lem}
\begin{proof} (i) By the Sobolev embedding and Hardy-Littlewood-Sobolev inequality, for all $u\in E_{\lambda}^{+}\backslash\ \{0\}$ we have
$$
\aligned
J_{\lambda,\beta}(u)&=\frac{1}{2}\|u\|_{L_{\lambda,\beta}}^{2}-\frac{1}{2\cdot2_{\mu}^{\ast}}\int_{\mathbb{R}^N}\int_{\mathbb{R}^N}
\frac{|u(x)|^{2_{\mu}^{\ast}}|u(y)|^{2_{\mu}^{\ast}}}
{|x-y|^{\mu}}dxdy\\
&\geq \frac{1}{2}\|u\|_{L_{\lambda,\beta}}^{2}-\frac{1}{2\cdot2_{\mu}^{\ast}}C_{1}|u|_{2^{\ast}}^{2\cdot2_{\mu}^{\ast}}\\
&\geq \frac{1}{2}\|u\|_{L_{\lambda,\beta}}^{2}-C_{2}\|u\|_{L_{\lambda,\beta}}^{2\cdot2_{\mu}^{\ast}}.\\
\endaligned
$$
Since $2<2\cdot2_{\mu}^{\ast}$, we can choose some $r,\alpha>0$ such that $J_{\lambda,\beta}|_{S_{r}^{+}}(u)\geq\alpha$ and $J_{\lambda,\beta}|_{B_{r}^{+}}(u)\geq0$.

(ii) We only need to show
if $\mathcal{V}\subset E_{\lambda}^{+}\backslash\{0\}$ is a compact subset, then there exists $R>0$ such that $J_{\lambda,\beta}<0$ on $H_{w}\backslash B_{R}$ for every $w\in\mathcal{V}$.

As in \cite{SW}, we may assume that
$\|w\|_{L_{\lambda,\beta}}= 1$ for every $w\in\mathcal{V}$. Suppose by
contradiction that there exist $w_{n}\in\mathcal{V}$ and $u_{n}\in H_{w_{n}}$, $n\in\N$, such that $J_{\lambda,\beta}(u_{n})\geq0$ for all $n$ and $\|u_{n}\|_{L_{\lambda,\beta}}\rightarrow\infty$ as $n\rightarrow\infty$. Passing to a subsequence, we may assume that $w_{n}\rightarrow w_{0}\in E_{\lambda}^{+}$, $\|w_{0}\|_{L_{\lambda,\beta}}= 1$. Set $v_{n}=\frac{u_{n}}{\|u_{n}\|_{L_{\lambda,\beta}}}=t_{n}w_{n}+v_{n}^{-}$, then
\begin{equation}\label{f3}
0\leq\frac{J_{\lambda,\beta}(u_{n})}{\|u_{n}\|_{L_{\lambda,\beta}}^{2}}=\frac{1}{2}
(t_{n}^{2}-\|v_{n}^{-}\|_{L_{\lambda,\beta}}^{2})-\frac{1}{2\cdot2_{\mu}^{\ast}}\int_{\mathbb{R}^N}
\int_{\mathbb{R}^N}\frac{|u_{n}(x)|^{2_{\mu}^{\ast}-1}|v_{n}(x)||u_{n}(y)|^{2_{\mu}^{\ast}-1}|v_{n}(y)|}{|x-y|^{\mu}}dxdy.
\end{equation}
Hence $\|v_{n}^{-}\|_{L_{\lambda,\beta}}^{2}\leq t_{n}^{2}=1-\|v_{n}^{-}\|_{L_{\lambda,\beta}}^{2}$ and $\frac{1}{\sqrt{2}}\leq t_{n}\leq1$. So, for a subsequence, $t_{n}\rightarrow t_{0}>0$, $v_{n}\rightharpoonup v_{0}$ in $E_{\lambda}$ and $v_{n}(x)\rightarrow v_{0}(x)$ a.e. in $\mathbb{R}^N$. Hence $v_{0}=t_{0}w_{0}+v_{0}^{-}\neq0$ and, since $|u_{n}(x)|\rightarrow\infty$ if $v_{0}(x)\neq0$,
$$
\int_{\mathbb{R}^N}
\int_{\mathbb{R}^N}\frac{|u_{n}(x)|^{2_{\mu}^{\ast}-1}
|v_{n}(x)||u_{n}(y)|^{2_{\mu}^{\ast}-1}|v_{n}(y)|}{|x-y|^{\mu}}dxdy\rightarrow\infty,
$$
contrary to \eqref{f3}.
\end{proof}

Define
$$
c^{\star}:=\inf\limits_{w\in E_{\lambda}^{+}\backslash\{0\}}\max_{u\in H_{w}}J_{\lambda,\beta}(u).
$$
As a consequence of Lemma \ref{IP33} we have

\begin{cor}\label{IP3}
There exist $\alpha, C > 0$ such that $\alpha\leq c^{\star}<C$.
\end{cor}

Following Ackermann \cite{AC1}, for a fixed $u\in E_{\lambda}^{+}$ we introduce $\Phi_{u}:E_{\lambda}^{-}\rightarrow \mathbb{R}$ defined by
$$
\Phi_{u}(v)=J_{\lambda,\beta}(u+v).
$$
Let $\Psi(u):=\displaystyle\frac{1}{2\cdot2_{\mu}^{\ast}}\int_{\mathbb{R}^N}\int_{\mathbb{R}^N}
\frac{|u(x)|^{2_{\mu}^{\ast}}|u(y)|^{2_{\mu}^{\ast}}}
{|x-y|^{\mu}}dxdy$, by direct computation and $\mu<4$, we know
$$
\langle\Psi''(u)w,w\rangle\geq0
$$
for all $u, w \in E_{\lambda}$, and hence
$$
\langle\Phi_{u}''(v)w,w\rangle=\langle J_{\lambda,\beta}''(u+v)w,w\rangle=-\|w\|_{L_{\lambda,\beta}}^{2}-\langle\Psi''(u+v)w,w\rangle\leq-\|w\|_{L_{\lambda,\beta}}^{2}.
$$
In addition,
$$
\Phi_{u}(v)\leq \frac{1}{2}\|u\|_{L_{\lambda,\beta}}^{2}-\frac{1}{2}\|v\|_{L_{\lambda,\beta}}^{2}.
$$
Therefore $\Phi_{u}$ is strictly concave and $\lim_{\|v\|_{L_{\lambda,\beta}}\rightarrow\infty}\Phi_{u}(v)=-\infty$. From weak sequential
upper semicontinuity of $\Phi_{u}$, it follows that there is a unique strict maximum point $h(u)\in E_{\lambda}^{-}$ for $\Phi_{u}$, which is also the only critical point of $\Phi_{u}$ on $E_{\lambda}^{-}$. Thus $h(u)$ satisfies
\begin{equation}\label{f4}
\langle\Phi_{u}'(h(u)),v\rangle=0
\end{equation}
for all $v\in E_{\lambda}^{-}$, and
$$
v\neq h(u)\Leftrightarrow J_{\lambda,\beta}(u+v)<J_{\lambda,\beta}(u+h(u)).
$$

As [\cite{AC1}, Lemma 5.6], we have the following:

\begin{lem}\label{Eh}
(i) $h$ is $\mathbb{R}^N$-invariant, i.e. $h(a\ast u) = h(u)$ where $(a\ast u)(x):= u(x+a)$ for all
$a\in \mathbb{R}^N$. \\
(ii) $h\in \mathcal{C}^{1}(E_{\lambda}^{+},E_{\lambda}^{-})$ and $h(0)=0$.\\
(iii) $h$ is a bounded map.\\
(iv) If $u_n\rightharpoonup u$ in $E_{\lambda}^{+}$, then $h(u_n)-h(u_n-u)\rightarrow h(u)$ and $h(u_n)\rightharpoonup h(u)$. The same is true for $|h(u)|_{2}^{2}$.
\end{lem}

Define $\Upsilon:E_{\lambda}^{+}\rightarrow\mathbb{R}$ by
$$
\Upsilon(u)=J_{\lambda,\beta}(u+h(u))= \frac{1}{2}\|u\|_{L_{\lambda,\beta}}^{2}-\frac{1}{2}\|h(u)\|_{L_{\lambda,\beta}}^{2}-\Psi(u+h(u)).
$$
By Theorem 5.1 in \cite{AC1}, we know that the critical points of $\Upsilon$ and $J_{\lambda,\beta}$ are one to one correspondence
via the injective map $u\rightarrow u+h(u)$ from $E_{\lambda}^{+}$ into $E_{\lambda}$.

Let
$$
{\cal N}:=\{u\in E_{\lambda}^{+}\backslash\{0\}:\langle \Upsilon'(u),u\rangle=0\},
$$
and we define
$$
c^{\star\star}=\inf_{u\in{\cal N}}\Upsilon(u).
$$

\begin{lem}\label{IP5}
$c^{\star}=c^{\star\star}=c_{\lambda}$, where $c_{\lambda}$ is defined in \eqref{f1}.
\end{lem}
\begin{proof}
As in \cite{D}, given $w\in E_{\lambda}^{+}$, if $u=tw+v\in H_{w}$ with $J_{\lambda,\beta}(u)=\max_{z\in H_{w}}J_{\lambda,\beta}(z)$ then the restriction $J_{\lambda,\beta}|_{H_{w}}$ of $J_{\lambda,\beta}$ on $H_{w}$ satisfies $(J_{\lambda,\beta}|_{H_{w}})'(u)=0$ which implies $v=h(tw)$ and $\langle\Upsilon'(tw),tw\rangle=\langle J_{\lambda,\beta}'(u),tw\rangle=0$, i.e. $tw\in{\cal N}$. Thus $c^{\star}\geq c^{\star\star}$. On the other hand, if $e\in{\cal N}$ then $(J_{\lambda,\beta}|_{H_{e}})'(e+h(e))=0$ so $c^{\star}\leq \max_{z\in H_{e}}J_{\lambda,\beta}(z)=\Upsilon(e)$. Thus $c^{\star}\leq c^{\star\star}$ and similarly,  $c_{\lambda}\leq c^{\star\star}$. This proves $c^{\star}=c^{\star\star}$.

For any $w\in {\cal N}_{\lambda}$, we have $w^{+}+w^{-}\in E_{\lambda}^{+}\oplus E_{\lambda}^{-}$ and $w^{+}\neq0$. So $w\in H_{w^{+}}$. Combining this with the fact that $\langle J_{\lambda,\beta}'(w),w\rangle=0$ and the discussion before Lemma \ref{Eh}, we have
$$
J_{\lambda,\beta}(w)=\max_{u\in H_{w^{+}}}J_{\lambda,\beta}(u),
$$
that is $J_{\lambda,\beta}(w)\geq c^{\star}$ and so $c_{\lambda}\geq c^{\star}$. Together with the fact that $c_{\lambda}\leq c^{\star\star}$, we have $c_{\lambda}= c^{\star\star}$. This proves $c_{\lambda}=c^{\star\star}=c^{\star}$.
\end{proof}

\section{The proof of Theorem \ref{EXS4}}

Now, we prove that for $\lambda$ large enough, any Palais-Smale sequence is bounded. For this, we
define
$$
X_{NL}:=\{u:\mathbb{R}^N\rightarrow\mathbb{R};\|u\|_{NL}<+\infty\},
$$
where
$$
\|\cdot\|_{NL}:=\left(\int_{\mathbb{R}^N}\int_{\mathbb{R}^N}\frac{|\cdot|^{2_{\mu}^{\ast}}|\cdot|^{2_{\mu}^{\ast}}}
{|x-y|^{\mu}}dxdy\right)^{\frac{1}{2\cdot2_{\mu}^{\ast}}}.
$$
By Lemma 2.3 of \cite{GY}, we know $\|\cdot\|_{NL}$ defines a norm on $X_{NL}$ under which $X_{NL}$ is a Banach space. Moreover the Hardy-Littlewood-Sobolev inequality also implies that $H^{1}(\mathbb{R}^N)$ is continuously embedded in $X_{NL}$.

\begin{lem}\label{IP14}
If $\{u_{n}\}$
is a $(PS)_{c_{\lambda}}$ sequence for $J_{\lambda,\beta}$, then $\{u_{n}\}$ is bounded in $E_{\lambda}$.
\end{lem}
\begin{proof}
Let $\vartheta\in(\frac{1}{2\cdot2_{\mu}^{\ast}},\frac{1}{2})$. It follows from $\{u_{n}\}$
is a $(PS)_{c_{\lambda}}$ sequence that, for $n$ large enough, we have
$$\aligned
c_{\lambda}+o_{n}(1)\|u_{n}\|_{L_{\lambda,\beta}}&\geq J_{\lambda,\beta}(u_{n})-\vartheta\langle J_{\lambda,\beta}'(u_{n}),u_{n}\rangle\\
&=(\frac{1}{2}-\vartheta)\int_{\mathbb{R}^N}(|\nabla u_{n}|^{2}+(\lambda V(x)-\beta)|u_{n}|^{2})dx+(\vartheta-\frac{1}{2\cdot2_{\mu}^{\ast}})
\|u_{n}\|_{NL}^{2\cdot2_{\mu}^{\ast}}\\
&=(\frac{1}{2}-\vartheta)(\|u_{n}^{+}\|_{L_{\lambda,\beta}}^{2}-\|u_{n}^{-}\|_{L_{\lambda,\beta}}^{2})+(\vartheta-\frac{1}{2\cdot2_{\mu}^{\ast}})
\|u_{n}\|_{NL}^{2\cdot2_{\mu}^{\ast}},\\
\endaligned
$$
where $u_{n}=u_{n}^{+}+u_{n}^{-}\in E_{\lambda}^{+}\oplus E_{\lambda}^{-}$. It is then easy to verify that $\{u_{n}\}$ is bounded in $E_{\lambda}$ by using the fact that that $E_{\lambda}^{-}$ is finite dimensional and $\|\cdot\|_{NL}$ is a norm in $X_{NL}$. This completes the proof of Lemma \ref{IP14}.
\end{proof}

Enlarging $\lambda_{\beta}$ if necessary, we may assume that $\lambda_{\beta}\geq\beta/M_{0}$, thus
$$
\lambda M_{0}-\beta\geq0 \hspace{4.14mm}\mbox{for} \hspace{1.14mm}\mbox{all} \hspace{1.14mm}\lambda\geq\lambda_{\beta},
$$
where $M_{0}$ is given in Remark \ref{E0}.

\begin{Prop}\label{PS2} Suppose $\lambda\geq\lambda_{\beta}$ and $\{u_{n}\}$ is $(PS)_{c_{\lambda}}$ sequence of $J_{\lambda,\beta}$ with  $$c_{\lambda}<\frac{N+2-\mu}{4N-2\mu}S_{H,L}^{\frac{2N-\mu}{N+2-\mu}}.$$
Then there exists a subsequence of $\{u_{n}\}$ which
converge strongly in $E_{\lambda}$ a solution $u_{\lambda}$ of \eqref{CCE} such that $J_{\lambda,\beta}(u_{\lambda})=c_{\lambda}$.
\end{Prop}
\begin{proof}
By Lemma \ref{IP14}, we know that $\{u_{n}\}$ is bounded in $E_{\lambda}$. Similar to the proof of Proposition \ref{PS}, we can obtain that there exists a subsequence of $\{u_{n}\}$ which
converge strongly in $E_{\lambda}$ a solution $u_{\lambda}$ of \eqref{CCE} such that $J_{\lambda,\beta}(u_{\lambda})=c_{\lambda}$.
\end{proof}

\begin{lem}\label{IP15}
For $\lambda>\lambda_{\beta}$, we have $$c_{\lambda}<\frac{N+2-\mu}{4N-2\mu}S_{H,L}^{\frac{2N-\mu}{N+2-\mu}}.$$
\end{lem}
\begin{proof}
By the definition of $c_{\lambda}$ we know that $c_{\lambda}\leq c(\beta,\Omega)$, where $c(\beta,\Omega)$ is defined as in
\eqref{f2}. By Proposition \ref{MPE4}, we know that
$$c_{\lambda}<\frac{N+2-\mu}{4N-2\mu}S_{H,L}^{\frac{2N-\mu}{N+2-\mu}}$$
and we complete the proof.
\end{proof}

\begin{Prop}\label{IP16}
For $\lambda>\lambda_{\beta}$, there is a ground state solution $u_{\lambda}$ of \eqref{CCE} which achieves $c_{\lambda}$.
\end{Prop}
\begin{proof}
Let $\{w_{n}\}\subset{\cal N}$ be a minimization sequence: $\Upsilon(w_{n})\rightarrow c^{\star\star}$. By the Ekeland variational principle we can assume that $\{w_{n}\}$ is, in addition, a $(PS)_{c^{\star\star}}$ sequence for $\Upsilon$ on ${\cal N}$. A standard argument shows that $\{w_{n}\}$ is in fact a $(PS)_{c^{\star\star}}$ sequence for $\Upsilon$ on $E_{\lambda}^{+}$ (see, e.g., \cite{Wi}). Then $\{u_n=w_n+h(w_n)\}$ is
a $(PS)_{c_{\lambda}}$ sequence for $J_{\lambda,\beta}$ on $E_{\lambda}$. By Proposition \ref{PS2} and Lemma \ref{IP15}, we have that there is a ground state solution $u_{\lambda}$ of \eqref{CCE} which achieves $c_{\lambda}$.
\end{proof}

In the following, we come to give the asymptotic behavior of the
ground state solutions of \eqref{CCE} as $\lambda$ goes to infinity.

\begin{Prop}\label{IP17}
$\lim_{\lambda\rightarrow+\infty}c_{\lambda}=c(\beta,\Omega)$ and for any sequence $\{\lambda_{n}\}(\lambda_{n}\rightarrow+\infty)$, up to a
subsequence $u_{\lambda_{n}}\rightarrow u$ strongly in $H^{1}(\mathbb{R}^{N})$. Here $u$ is a ground state solution of \eqref{LCCE} which achieves $c(\beta,\Omega)$.
\end{Prop}
\begin{proof}
For $u\in H_{0}^{1}(\Omega)$, we have ${\cal N}_{\beta,\Omega}\subset{\cal N}_{\lambda}$. Thus by the definition of $c_{\lambda}$ and $c(\beta,\Omega)$, it is easy to see that $c_{\lambda}\leq c(\beta,\Omega)$ for $\lambda\geq \lambda_{\beta}$. On
the other hand, it is not difficult to check that $c_{\lambda}$ is nondecreasing as $\lambda$ growth. Thus we may assume that $\lim_{\lambda\rightarrow+\infty}c_{\lambda}=\kappa\leq c(\beta,\Omega)$ which implies for any sequence $\{\lambda_{n}\}(\lambda_{n}\rightarrow+\infty)$, $c_{\lambda_{n}}\rightarrow\kappa\leq c(\beta,\Omega)$. We assume that $u_{n}$ is such that $c_{\lambda_{n}}$ is achieved,
by Lemma \ref{IP14}, $\{u_{n}\}$ is bounded in $E_{\lambda_{n}}$ and thus is also bounded in $H^{1}(\mathbb{R}^{N})$. As a result, we have
$$
u_{n}\rightharpoonup u\ \ \mbox{weakly in}\ \ H^{1}(\mathbb{R}^{N}),
$$
$$
u_{n}\rightarrow u\ \ \mbox{strongly in}\ \ L_{loc}^{q}(\mathbb{R}^{N})\ \ \mbox{for}\ \ 2\leq q<2^{\ast}
$$
and
$$
u_{n}\rightarrow u\ \ \mbox{a.e. in}\ \ \mathbb{R}^{N}.
$$
We claim that $u|_{\Omega^{c}}=0$, where $\Omega^{c}:=\{x|x\in\mathbb{R}^{N}\backslash\Omega\}$. Indeed, if not, there exists a compact subset $D_{1}\subset\Omega^{c}$ with $dist\{D_{1},\partial\Omega\}>0$ such that $u|_{D_{1}}\neq0$ and
$$
\int_{D_{1}}u_{n}^{2}dx\rightarrow\int_{D_{1}}u^{2}dx>0.
$$
Moreover, there exists $\epsilon_{0}>0$ such that $V(x)\geq\epsilon_{0}$ for any $x\in D_{1}$.

By the choice of $\{u_{n}\}$, we have
$$
0=\langle J_{\lambda,\beta}'(u_{n}),u_{n}\rangle
=\int_{\mathbb{R}^{N}}(|\nabla u_{n}|^{2}+(\lambda_{n} V-\beta) u_{n}^{2})dx-\int_{\mathbb{R}^N}\int_{\mathbb{R}^N}
\frac{|u_{n}(x)|^{2_{\mu}^{\ast}}|u_{n}(y)|^{2_{\mu}^{\ast}}}
{|x-y|^{\mu}}dxdy,
$$
hence for $n$ large
$$\aligned
J_{\lambda,\beta}(u_{n})&=(\frac{1}{2}-\frac{1}{2\cdot2_{\mu}^{\ast}})\int_{\mathbb{R}^{N}}|\nabla u_{n}|^{2}+(\lambda_{n} V(x)-\beta) u_{n}^{2}dx\\
&\geq(\frac{1}{2}-\frac{1}{2\cdot2_{\mu}^{\ast}})\int_{D_{1}}(\lambda_{n}\epsilon_{0} -\beta) u_{n}^{2}dx\rightarrow+\infty
\endaligned$$
as $n\rightarrow\infty$. This contradiction shows that $u|_{\Omega^{c}}=0$. By the smooth assumption on the boundary $\partial\Omega$ we indeed have $u\in H_{0}^{1}(\Omega)$.

Now we prove that
\begin{equation}\label{g1}
u_{n}\rightarrow u\ \ \mbox{strongly in}\ \ X_{NL}.
\end{equation}
We take $v_{n}:=u_{n}-u$ and suppose on the contrary that \eqref{g1} is not true, then up to a subsequence, we may assume that
$$
\lim_{n\rightarrow\infty}\int_{\mathbb{R}^N}\int_{\mathbb{R}^N}
\frac{|v_{n}(x)|^{2_{\mu}^{\ast}}|v_{n}(y)|^{2_{\mu}^{\ast}}}
{|x-y|^{\mu}}dxdy\rightarrow b>0.
$$
By a similar argument as the proof of Proposition \ref{PS}, we can show that $b\geq S_{H,L}^{\frac{2N-\mu}{N-\mu+2}}$ which implies
that $\kappa=\lim_{n\rightarrow\infty}c_{\lambda_{n}}\geq \frac{N+2-\mu}{4N-2\mu}S_{H,L}^{\frac{2N-\mu}{N+2-\mu}}$. This contradicts with $\kappa<c(\beta,\Omega)<\frac{N+2-\mu}{4N-2\mu}S_{H,L}^{\frac{2N-\mu}{N+2-\mu}}$. Namely we proved that \eqref{g1} holds.

From the fact that $u_{n}$ is the solutions of \eqref{CCE} with $\lambda$ replaced by $\lambda_{n}$, we have
$$
\int_{\mathbb{R}^{N}}(\nabla u_{n}\nabla\varphi+(\lambda_{n} V-\beta)u_{n}\varphi )dx
=\int_{\mathbb{R}^N}\int_{\mathbb{R}^N}
\frac{|u_{n}(x)|^{2_{\mu}^{\ast}}|u_{n}(y)|^{2_{\mu}^{\ast}-2}u_{n}(y)\varphi(y)}
{|x-y|^{\mu}}dxdy
$$
for any $\varphi\in E$. If $\varphi\in H_{0}^{1}(\Omega)$ then $\lambda_{n}\displaystyle\int_{\mathbb{R}^N} Vu_{n}\varphi dx=0$ for all $n$. Letting $n\rightarrow\infty$ we obtain
$$
\int_{\mathbb{R}^{N}}(\nabla u\nabla\varphi
-\beta u\varphi) dx=\int_{\mathbb{R}^N}\int_{\mathbb{R}^N}
\frac{|u(x)|^{2_{\mu}^{\ast}}|u(y)|^{2_{\mu}^{\ast}-2}u(y)\varphi(y)}
{|x-y|^{\mu}}dxdy
$$
for any $\varphi\in H_{0}^{1}(\Omega)$. So, $u$ is a solution of \eqref{LCCE}.
Since $V(x)=0$ for $x\in\Omega$, we get
\begin{equation}\label{g2}
\langle L_{\lambda_{n},\beta}u_{n},u_{n}\rangle=\langle L_{0,\beta} u,u\rangle+\langle L_{\lambda_{n},\beta}v_{n},v_{n}\rangle,
\end{equation}
where $v_{n}=u_{n}-u$. Since $\{u_{n}\}$ is a sequence of solutions of \eqref{CCE} and $u$ is a solution of \eqref{LCCE}, by \eqref{g1} we can get
\begin{equation}\label{g3}
\langle L_{\lambda_{n},\beta}v_{n},v_{n}\rangle=o_{n}(1).
\end{equation}
Thus, from \eqref{g2} we get
$$
\langle L_{\lambda_{n},\beta}u_{n},u_{n}\rangle\rightarrow\langle L_{0,\beta} u,u\rangle
$$
as $n\rightarrow\infty$, that is
$$
\int_{\mathbb{R}^{N}}(|\nabla u_{n}|^{2}+(\lambda_{n} V(x)-\beta) u_{n}^{2})dx\rightarrow
\int_{\mathbb{R}^{N}}(|\nabla u|^{2}-\beta u^{2})dx
$$
as $n\rightarrow\infty$. By Lemma \ref{EMB}, we know $u_{n}\rightarrow u$ in $L^{2}(\mathbb{R}^N)$ and so $$
\int_{\mathbb{R}^{N}}(|\nabla u_{n}|^{2}+\lambda_{n} V(x) u_{n}^{2})dx\rightarrow
\int_{\mathbb{R}^{N}}|\nabla u|^{2}dx
$$
as $n\rightarrow\infty$. It follows from
$$
\int_{\mathbb{R}^N}|\nabla u|^{2}dx\leq\lim_{n\rightarrow\infty}\int_{\mathbb{R}^N}|\nabla u_{n}|^{2}dx
$$
that
$$
\int_{\mathbb{R}^{N}}|\nabla u_{n}|^{2}dx\rightarrow
\int_{\mathbb{R}^{N}}|\nabla u|^{2}dx
$$
as $n\rightarrow\infty$. Combining this with the fact that $u_{n}\rightarrow u$ in $L^{2}(\mathbb{R}^N)$, we have
$$
u_{n}\rightarrow u\ \ \mbox{strongly in}\ \ H^{1}(\mathbb{R}^{N}).
$$
By the
definition of $c(\beta,\Omega)$ we have $J_{\beta,\Omega}(u)\geq c(\beta,\Omega)$. On the other hand, by the strong convergence of
$u_n$, we have
$$\aligned
J_{\beta,\Omega}(u)&=(\frac{1}{2}-\frac{1}{2\cdot2_{\mu}^{\ast}})\int_{\Omega}
\int_{\Omega}\frac{|u(x)|^{2_{\mu}^{\ast}}|u(y)|^{2_{\mu}^{\ast}}}
{|x-y|^{\mu}}dxdy\\
&=(\frac{1}{2}-\frac{1}{2\cdot2_{\mu}^{\ast}})\lim_{n\rightarrow\infty}\int_{\mathbb{R}^N}
\int_{\mathbb{R}^N}\frac{|u_{n}(x)|^{2_{\mu}^{\ast}}|u_{n}(y)|^{2_{\mu}^{\ast}}}
{|x-y|^{\mu}}dxdy\\
&=\lim_{n\rightarrow\infty}J_{\lambda,\beta}(u_{n})\\
&=\lim_{n\rightarrow\infty}c_{\lambda_{n}}=\kappa\leq c(\beta,\Omega).
\endaligned$$
Thus we proved that $J_{\beta,\Omega}(u) =\kappa= c(\beta,\Omega)$. Namely $u$ is indeed a ground state solution of \eqref{LCCE} which achieves $c(\beta,\Omega)$ and thus the proof of Proposition \ref{IP17} is completed.
\end{proof}

\noindent
{\bf Proof of Theorem \ref{EXS4}.} This is a direct results of Proposition \ref{IP16} and Proposition \ref{IP17}.
$\hfill{} \Box$

\end{document}